\documentclass[10pt]{article}

\usepackage{amssymb,amsmath,amsthm}
\usepackage{amsfonts}
\usepackage{graphicx,subfigure,epstopdf}
\usepackage[usenames]{color}
\usepackage{url}
\usepackage{algorithm,algorithmic}
\usepackage{dsfont,verbatim}
\graphicspath{{./pics/}}
\allowdisplaybreaks

\newtheorem{theorem}{Theorem}[section]
\newtheorem{lemma}{Lemma}[section]

\newtheorem{remark}{Remark}[section]

\newtheorem{condition}{Condition}[section]

\numberwithin{equation}{section}

\newcommand{\al}{\alpha}

\def\Dal{{\partial_t^\al}}
\def\bPtau{\bar\partial_\tau}

\title{\bf Correction of high-order BDF convolution quadrature for fractional evolution equations} 

\author{Bangti Jin\thanks{Department of Computer Science, University College London, Gower Street, London WC1E 6BT, UK
(\texttt{b.jin@ucl.ac.uk})}
\and Buyang Li\thanks{Department of Applied Mathematics, The Hong Kong Polytechnic University, Hung Hom, Hong Kong.
(\texttt{buyang.li@polyu.edu.hk})}
\and Zhi Zhou\thanks{Department of Applied Physics and Applied Mathematics,
Columbia University, 500 W. 120th Street, New York, NY 10027, USA (\texttt{zz2393@columbia.edu})}}
\date{\today}

\begin{document}

\maketitle

\begin{abstract}
We develop proper correction formulas at the starting $k-1$ steps
to restore the desired $k^{\rm th}$-order
convergence rate of the $k$-step BDF convolution quadrature for discretizing evolution equations involving a fractional-order derivative in time.
The desired $k^{\rm th}$-order convergence rate can be achieved even if the source term is not compatible with the initial data, which is allowed to be nonsmooth. We provide complete error estimates for the subdiffusion case $\alpha\in (0,1)$, and sketch the proof for the diffusion-wave case $\alpha\in(1,2)$. Extensive numerical examples are provided to illustrate the effectiveness of the proposed scheme.
\end{abstract}


\pagestyle{myheadings}
\thispagestyle{plain}

\section{Introduction} \label{sec:intro}

We are interested in the convolution quadrature (CQ) generated by high-order backward difference formulas (BDFs) for solving
the fractional-order evolution equation (with $0<\alpha<1$)
\begin{equation}\label{eqn:fde}
\left\{
\begin{aligned}
   & \Dal (u(t)-v) -  Au(t)= f(t),\quad  0<t<T,\\
   & u(0)=v ,
\end{aligned}
\right.
\end{equation}
where $f$ is a given function, and $\partial_t^\alpha u$ denotes the left-sided Riemann-Liouville fractional time derivative of order $\alpha$, defined by (cf. \cite{KilbasSrivastavaTrujillo:2006})
\begin{equation}\label{McT}
   \partial_t^\alpha u(t):= \frac{1}{\Gamma(1-\al)} \frac{d}{dt} \int_0^t(t-s)^{-\al}u(s)\, ds,
\end{equation}
where $\Gamma(z):=\int_0^\infty s^{z-1}e^{-s}ds$ is the Gamma function.
Under the initial condition $u(0)=v$, the Riemann-Liouville fractional derivative $\Dal (u-v)$ in the
model \eqref{eqn:fde} is identical with the usual Caputo fractional derivative \cite[pp. 91]{KilbasSrivastavaTrujillo:2006}.

In the model \eqref{eqn:fde}, the operator $A$ denotes either the Laplacian $\Delta$ on a polyhedral
domain $\Omega\subset \mathbb R^d\,(d=1,2,3)$ with a homogenous Dirichlet boundary condition, or its discrete approximation $\Delta_h$ by Galerkin finite element method. Thus the operator $A$ satisfies the following
resolvent estimate (cf. \cite[Example 3.7.5 and Theorem 3.7.11]{ABHN} and \cite{Thomee:2006})
\begin{equation}\label{eqn:resolvent-estimate}
  \| (z -A)^{-1} \|_{{L^2(\Omega)\rightarrow L^2(\Omega)}}\le c_\phi z^{-1},  \quad \forall z \in \Sigma_{\phi},
\end{equation}
for all $\phi\in (\pi/2,\pi)$, where $\Sigma_\theta:=\{z\in\mathbb{C}\setminus\{0\}: |\arg z|< \theta\}$
is a sector of the complex plane $\mathbb{C}$. The model \eqref{eqn:fde}
covers a broad range of applications related to anomalous diffusion discovered in the past two decades, e.g.,
conformational dynamics of protein molecules, contaminant transport in complex geological formations
and relaxation in polymer systems; see \cite{SokolovKlafterBlumen:2002}.

There has been much recent interest in developing high-order schemes for problem \eqref{eqn:fde}, especially
spectral methods \cite{LiXu:2009,ZayernouriKarniadakis:2014, Zayernouri:2015uni, ChenXuHesthaven:2015,ChenShenWang:2016}
and discontinuous Galerkin \cite{Deng:2015local, Mustapha:2015, MustaphaAbdallahFurati:2014, Mustapha:2014hp}. In this work, we develop robust high-order schemes
based on CQs generated by high-order BDFs. The CQ developed by Lubich \cite{Lubich:1986,Lubich:1988,Lubich:2004} provides a flexible framework
for constructing high-order methods to discretize the fractional derivative $\partial_t^\alpha u$. By its very
construction, it inherits the stability properties of linear multistep methods, which greatly facilitates the
analysis of the resulting numerical scheme, in a way often strikingly opposed to standard quadrature formulas
\cite[pp. 504]{Lubich:2004}. Hence, it has been widely applied to discretize the model \eqref{eqn:fde} and its variants,
especially the CQ generated by BDF1 and BDF2 (with BDF$k$ denoting BDF of order $k$).
In the literature, the CQ generated by BDF1 is commonly known as the Gr\"unwald-Letnikov formula.

By assuming that the solution is sufficiently smooth, which is equivalent to assuming smoothness of the initial
data $v$ and imposing certain \textit{compatibility} conditions on the source term $f$ at $t=0$, the stability and
convergence of the numerical solutions of fractional evolution equations have been investigated in
\cite{CuestaPalencia:2003,GaoSunSun:2015,WangVong:2014,YusteAcedo:2005,ZengLiLiuTurner:2013}.
In general, if the source term $f$ is not compatible with the given initial data, the solution $u$ of the model
\eqref{eqn:fde} will exhibit weak singularity at $t=0$, which will deteriorate the convergence rate of the
numerical solutions. This has been widely recognized in fractional ODEs \cite{DiethelmFordFreed:2004,
GalucioDeuMengu:2006} and PDEs \cite{CuestaLubichPalencia:2006,JinLazarovZhou:SISC2016,Sanz-Serna1988}. In particular, direct implementation of the CQ generated by high-order BDFs for discretizing the fractional evolution equations generally only yields first-order accuracy. To restore the theoretical
rate $O(\tau^k)$ of BDF$k$, two different strategies have been proposed.

For fractional ODEs, one idea is to use starting weights \cite{Lubich:1986} to correct the CQ in discretizing
the fractional time derivative, cf. \eqref{eqn:CQ} below:
\begin{equation*}
   \bar\partial_\tau^\alpha \varphi^n = \frac{1}{\tau^\alpha}\sum_{j=0}^n b_{n-j}\varphi^j + \sum_{j=0}^M w_{n,j}\varphi^j.
\end{equation*}
The starting term $\sum_{j=0}^M w_{n,j}u_j$ is to capture all leading singularities so
as to recover a \textit{uniform} $O(\tau^k)$ rate of the scheme, where $M\in\mathbb{N}$ and the
weights $w_{n,j}$ generally depend on both $\alpha$ and $k$. This approach works well for
fractional evolution ODEs, however, the extension of this approach to fractional evolution PDEs
relies on expanding the solution into power series of $t$, which requires imposing certain
compatibility conditions on the source term.

The second idea is to split the source term $f$ into $f(t)=f(0) + (f(t)-f(0))$ and to approximate $f(0)$ by
$\bar\partial_\tau \partial_t^{-1}f(0)$, with a similar treatment of the initial data $v$. This leads to
a corrected BDF2 at the first step and restores the $O(\tau^2)$ accuracy for any fixed $t_n>0$.
The idea was first introduced in \cite{LubichSloanThomee:1996} for solving a variant of
formulation \eqref{eqn:fde} in the diffusion-wave case and then systematically developed in
\cite{CuestaLubichPalencia:2006} for BDF2, and was recently extended to the model \eqref{eqn:fde} in
\cite{JinLazarovZhou:SISC2016} for both subdiffusion and diffusion-wave cases. Higher-order extension of
this approach is possible, but is still not available in the literature.

The goal of this work is to develop robust high-order BDFs for fractional evolution equations along the second strategy \cite{CuestaLubichPalencia:2006,JinLazarovZhou:SISC2016}. Instead of extending this strategy to each
high-order BDF method, separately, we develop a systematic strategy for correcting initial steps for high-order BDFs,
based on a few simple criteria, cf. \eqref{crit:mu} and \eqref{crit:b} for the model \eqref{eqn:fde}. These criteria
emerge naturally from solution representations, and are purely algebraic in nature and straightforward to construct.
The explicit correction coefficients will be given for BDFs up to order $6$. For BDF$k$, the correction is only
needed at the starting $k-1$ steps and thus the resulting scheme is easy to implement.

We develop proper corrections for high-order BDFs for both subdiffusion, i.e., $\alpha\in(0,1)$, and diffusion wave, i.e.,
$\alpha\in(1,2)$. It is noteworthy that for $\alpha\in(1,2)$, high-order BDFs can be either
unconditionally or conditionally stable, depending on the fractional order $\alpha$, and in
the latter case, an explicit CFL condition on the time step size $\tau$ is given. Theoretically, the corrected BDF$k$ achieves
the $k^{\rm th}$-order accuracy at any fixed time $t=t_n$ (when $t_n$ is bounded from below),
and the error bound depends only on data regularity, without assuming any compatibility conditions
on the source term or extra regularity on the solution  (cf. Theorems \ref{thm:conv} and \ref{thm:conv-dw}).
These results are {supported} by the numerical experiments in Section \ref{sec:numerics}.

The rest of the paper is organized as follows. In Section \ref{sec:scheme} we develop the correction
for the subdiffusion case, including the motivations of the algebraic criteria for choosing the correction coefficients.
The extension of the approach to the diffusion wave case is given in Section \ref{sec:diff-wave}. Numerical results are
presented in Section \ref{sec:numerics} to illustrate the efficiency and robustness of the corrected schemes. Appendix
\ref{app:correction-Lubich} gives an alternative interpretation of our correction method in terms of Lubich's convolution
quadrature for operator-valued convolution integrals. Some lengthy
proofs are given in Appendices B, C and D. Throughout, the notation $c$ denotes a generic positive constant,
whose value may differ at each occurrence, but it is always independent of the time step size $\tau$ and the solution $u$.

\section{BDFs for Subdiffusion and its Correction}\label{sec:scheme}

Let $\{t_n=n\tau\}_{n=0}^N$ be a uniform partition of the interval $[0,T]$,
with a time step size $\tau=T/N$. The CQ generated by BDF$k$, $k=1,\ldots,6$,
approximates the fractional derivative
$\partial_t^\alpha \varphi(t_n)$ by
\begin{equation}\label{eqn:CQ}
\bar\partial_\tau^\alpha \varphi^n:=
\frac{1}{\tau^\alpha}\sum_{j=0}^n b_j \varphi^{n-j},
\end{equation}
with $\varphi^n=\phi(t_n)$, where the weights $\{b_j\}_{j=0}^\infty$ are the coefficients in the power series expansion
\begin{equation}\label{eqn:delta}
\delta_\tau(\zeta)^\alpha
=\frac{1}{\tau^\alpha}\sum_{j=0}^\infty b_j\zeta ^j \quad \mbox{with}\quad \delta_\tau(\zeta ):=
\frac{1}{\tau}
\sum_{j=1}^k \frac {1}{j} (1-\zeta )^j.
\end{equation}
Below we often write $\delta(\zeta)=\delta_1(\zeta)$. The coefficients $b_j$ can be computed efficiently
by the fast Fourier transform \cite{Podlubny:1988,Sousa:2012} or recursion \cite{Wu:2014}. Correspondingly,
the BDF for solving \eqref{eqn:fde} seeks approximations $U^n$, $n=1,\dots,N$, to the exact solution $u(t_n)$ by
\begin{equation}\label{eqn:BDF-CQ-0}
\bPtau^\alpha (U-v)^n  -  A U^n = f(t_n) .
\end{equation}
If the solution $u$ is smooth and has sufficiently many vanishing derivatives at $0$, then
$U^n$ converges at a rate of $O(\tau^k)$ \cite{Lubich:1986,Lubich:2004}. However, it generally only
exhibits a first-order accuracy when solving fractional evolution equations, due to the weak
solution singularity at $0$, even if the initial data $v$ and source term $f$ are smooth
\cite{SakamotoYamamoto:2011}. This has been observed numerically
\cite{CuestaLubichPalencia:2006,JinLazarovZhou:SISC2016}. For $\alpha=1$, BDF$k$
is known to be $A(\vartheta_k)$-stable with angle $\vartheta_k= 90^\circ$, $90^\circ$, $86.03^\circ$,
$73.35^\circ$, $51.84^\circ$, $17.84^\circ$ for $k = 1,2,3,4,5,6$, respectively \cite[pp. 251]{HairerWanner:1996}.

To restore the $k^{\rm th}$-order accuracy, we correct BDF$k$ at the starting $k-1$ steps by
(as usual, the summation disappears if the upper index is smaller than the lower one)
\begin{equation}\label{eqn:BDF-CQ}
\begin{aligned}
&\bPtau^\alpha (U-v)^n  - A U^n = a_n^{(k)} (A v+f(0))+f(t_n)+
\sum_{\ell=1}^{k-2} b_{\ell,n}^{(k)}\tau^{\ell} \partial_t^{\ell}f(0) ,
&&1\le n\le k-1,\\
&\bPtau^\alpha (U-v)^n  -  A U^n = f(t_n) , &&k\le n\le N.
\end{aligned}
\end{equation}
where $a_n^{(k)}$ and $b_{\ell,n}^{(k)}$ are coefficients to be determined below. They are
constructed so as to improve the accuracy of the overall scheme to $O(\tau^k)$ for a general
initial data $v\in D(A)$ and a possibly incompatible right-hand side $f$.
The only difference between \eqref{eqn:BDF-CQ} and the standard scheme \eqref{eqn:BDF-CQ-0}
lies in the correction terms at the starting $k-1$ steps. Hence, the proposed scheme \eqref{eqn:BDF-CQ} is easy to implement.

\begin{remark}\label{rmk:f-ell}
In the scheme \eqref{eqn:BDF-CQ}, the derivative $\partial_t^{\ell}f(0)$ may be replaced by its
$(k-\ell-1)$-order finite difference approximation $f^{(\ell)}$, without sacrificing its accuracy.
\end{remark}

\begin{remark}\label{rmk:correction-Lubich}
The correction in \eqref{eqn:BDF-CQ} is minimal in the sense that there is no
other correction scheme which modifies only the $k-1$ starting steps
while retaining the $O(\tau^k)$ convergence. This does not
rule out corrections with more starting steps. We give an interesting correction  closely related
to \eqref{eqn:BDF-CQ} in Appendix \ref{app:correction-Lubich}.
\end{remark}

\subsection{Derivation of the correction criteria}
Now we derive the criteria for choosing the coefficients $a_j^{(k)}$ and
$b_{\ell,j}^{(k)}$, cf. \eqref{crit:mu} and \eqref{crit:b}, using Laplace transform and its
discrete analogue, the generating function \cite{LubichSloanThomee:1996,Thomee:2006}. We denote by
$~~\widehat{}~~$ taking Laplace transform, and for a given sequence $(f^n)_{n=0}^\infty$,
denote by $\widetilde f(\zeta )$ the generating function, which is defined by
$\widetilde f(\zeta ) := \sum_{n=0}^\infty f^n\zeta ^n.$
First we split the right hand side $f$ into
\begin{equation}\label{decomp-rhs}
 f(t)=  f(0) + \sum_{\ell=1}^{k-2} \frac{t^{\ell} }{\ell !}\partial_t^{\ell}f(0) + R_{k},
\end{equation}
and $R_k$ is the corresponding local truncation error, given by
\begin{equation}\label{eqn:Rk}
R_{k}
= f(t)-  f(0)- \sum_{\ell=1}^{k-2} \frac{t^{\ell} }{\ell !}\partial_t^\ell f(0)
=\frac{t^{k-1}}{(k-1)!}\partial_t^{k-1}f(0)
+\frac{t^{k-1}}{(k-1)!}*\partial_t^{k}f,
\end{equation}
where $\ast$ denotes Laplace convolution. Thus the function $w(t):=u(t)-v$ satisfies
\begin{equation}
    \Dal w -Aw = A v +f(0)
    + \sum_{\ell=1}^{k-2} \frac{t^{\ell}}{\ell !} \partial_t^\ell f (0) + R_{k},
\end{equation}
with $w(0)=0$. Since $w(0)=0$, the identity $\widehat{\partial_t^\alpha w}(z)=z^\alpha \widehat{w}(z)$
holds \cite[Remark 2.8, pp. 84]{KilbasSrivastavaTrujillo:2006}, and thus by Laplace transform, we obtain
\begin{equation*}
    z^\al \widehat w(z)  -A\widehat w(z) = z^{-1} (Av  +f(0))+
    \sum_{\ell=1}^{k-2} \frac{1}{z^{\ell+1}} \partial_t^\ell f (0) + \widehat R_{k}(z).
\end{equation*}
By inverse Laplace transform, the function $w(t)$ can be readily represented by
\begin{equation}\label{eqn:semisol}
\begin{aligned}
  w(t)=  &\frac{1}{2\pi \mathrm{i}}\int_{\Gamma_{\theta,\delta}}
  e^{zt}K(z)\big(Av  +f(0)\big)dz
  +\frac{1}{2\pi \mathrm{i}}\int_{\Gamma_{\theta,\delta}}
  e^{zt} zK(z)\bigg(\sum_{\ell=1}^{k-2} \frac{1}{z^{\ell+1}} \partial_t^\ell f (0)
    + \widehat R_{k}(z)\bigg)dz,
    \end{aligned}
\end{equation}
with the kernel function
\begin{equation}\label{eqn:kernel}
   K(z)= z^{-1}(z^\al-A)^{-1}.
\end{equation}
In the representation \eqref{eqn:semisol}, the contour $\Gamma_{\theta,\delta}$ is defined by
\begin{equation*}
  \Gamma_{\theta,\delta}=\left\{z\in \mathbb{C}: |z|=\delta, |\arg z|\le \theta\right\}\cup
  \{z\in \mathbb{C}: z=\rho e^{\pm\mathrm{i}\theta}, \rho\ge \delta\},
\end{equation*}
oriented with an increasing imaginary part.
Throughout, we choose the angle $\theta$ such that $ \pi/2 < \theta < \min(\pi,\pi/\al)$ and hence
$z^{\al} \in \Sigma_{\theta'}$ with $ \theta'=\al\theta< \pi$ for all $z\in\Sigma_{\theta}$.
By the resolvent estimate \eqref{eqn:resolvent-estimate}, there exists a constant $c$ which depends only on $\theta$ and $\al$ such that
\begin{equation}\label{eqn:resol}
  \| (z^{\al}-A)^{-1} \| \le cz^{-\al} \quad \mbox{and}\quad \| K(z)\|\le c|z|^{-1-\alpha},  \quad \forall z \in \Sigma_{\theta}.
\end{equation}

Next, we give a representation of the discrete solution $W^n:=U^n-v$, which
follows from lengthy but simple computations, cf. Appendix \ref{app:sol-rep}.
\begin{theorem}\label{thm:solurep}
The discrete solution $W^n:=U^n-v$ is represented by
\begin{equation}\label{eqn:Rep-Wh}
\begin{aligned}
  W^n&=
   \frac{1}{2\pi\mathrm{i} }\int_{\Gamma^\tau_{\theta,\delta}}e^{zt_{n}}
  \mu(e^{-z\tau}) K( \delta_\tau(e^{-z\tau}))(A v+ f(0))\,dz\\
  &+\frac{1}{2\pi\mathrm{i} }\int_{\Gamma^\tau_{\theta,\delta}}e^{zt_{n}}
  \delta_\tau(e^{-z\tau})K(\delta_\tau(e^{-z\tau}))
  \sum_{\ell=1}^{k-2}\bigg(\frac{\gamma_\ell(e^{-z\tau})}{\ell !}
  +\sum_{j=1}^{k-1} b_{\ell,j}^{(k)}e^{-zt_j}\bigg)\tau^{\ell+1} \partial_t^\ell f (0)\,dz \\
&+\frac{1}{2\pi\mathrm{i} }\int_{\Gamma^\tau_{\theta,\delta}}e^{zt_{n}}
  \delta_\tau(e^{-z\tau})K( \delta_\tau(e^{-z\tau}))\tau \widetilde  R_{k}(e^{-z\tau})
\,dz ,
  \end{aligned}
\end{equation}
with the contour
$\Gamma_{\theta,\delta}^\tau :=\{ z\in \Gamma_{\theta,\delta}:|\Im(z)|\le {\pi}/{\tau} \}$
{\rm(}oriented with an increasing imaginary part{\rm)},
where the functions $\mu(\zeta ) $ and $\gamma_\ell(\zeta )$ are respectively defined by
\begin{equation}\label{eqn:gamma-mu}
  \mu(\zeta ) = \delta(\zeta) \bigg( \frac{\zeta }{1-\zeta } + \sum_{j=1}^{k-1}a_j^{(k)} \zeta ^j\bigg)
  \quad \mbox{and}\quad \gamma_\ell(\zeta )=\bigg(\zeta  \frac{d}{d\zeta }\bigg)^\ell\frac{1}{1-\zeta }.
\end{equation}
\end{theorem}

By comparing the kernel functions in \eqref{eqn:semisol} and \eqref{eqn:Rep-Wh},
we deduce that in order to have $O(\tau^k)$ accuracy, the following three conditions should be satisfied for $z\in\Gamma_{\theta,\delta}^\tau$:
\begin{equation*}
  \begin{aligned}
    &|\delta_\tau(e^{-z\tau})-z|\leq c|z|^{k+1}\tau^k,\qquad
     |\mu(e^{-z\tau})-1| \leq c|z|^k\tau^k,\quad \\
    & \bigg|\bigg(\frac{\gamma_\ell(e^{-z\tau})}{\ell !}
  +\sum_{j=1}^{k-1} b_{\ell,j}^{(k)}e^{-zt_j} \bigg)\tau^{\ell+1}- \frac{1}{z^{\ell+1}}\bigg| \leq c|z|^{k-\ell-1}\tau^{k}.
  \end{aligned}
\end{equation*}
Note that for BDF$k$, the estimate $|\delta_\tau(e^{-z\tau})-z|\leq c|z|^{k+1}\tau^k$ holds automatically (cf. Lemma \ref{lem:delta} in Appendix \ref{app:sol-rep}). It suffices to impose the following algebraic criteria (changing $e^{-z\tau}$ to $\zeta$ and $z\tau $ to $1-\zeta$): for BDF$k$, choose the coefficients $\{a_j^{(k)}\}_{j=1}^{k-1}$ and $\{b_{\ell,j}^{(k)}\}_{j=1}^{k-1}$ such that
\begin{align}
  |\mu(\zeta )-1| & \leq c|1-\zeta |^k,\label{crit:mu}\\
  \bigg|\frac{\gamma_\ell(\zeta )}{\ell !}
  + \sum_{j=1}^{k-1} b_{\ell,j}^{(k)} \zeta ^j-\frac{1}{\delta(\zeta )^{\ell+1}} \bigg|
 &\le c |1-\zeta |^{k-\ell-1}, \quad \ell=1,\ldots,k-2, \label{crit:b}
\end{align}
where the functions $\mu(\zeta)$ and $\gamma_\ell(\zeta)$ are defined in \eqref{eqn:gamma-mu}. It can be
verified that for BDF$k$, $k=3,\ldots,6$, the leading singularities on the left
hand side of \eqref{crit:b} do cancel out, and thus the criterion can be satisfied.

\subsection{Computation of the coefficients $a_j^{(k)}$ and $b_{\ell,j}^{(k)}$}\label{ssec:coef-sub}
First we compute the coefficients $a_j^{(k)}$. To this end, we
rewrite $\sum_{j=1}^{k-1}a_j^{(k)}\zeta ^j$ as
\begin{align}\label{formula_aj}
\sum_{j=1}^{k-1}a_j^{(k)}\zeta ^j=\zeta \sum_{j=0}^{k-2}c_j (1-\zeta )^j.
\end{align}
Consequently, by writing $\zeta = 1- (1-\zeta)$, expanding the summation and collecting terms, we obtain (with the convention $c_{-2}=c_{-1}=0$)
\begin{align*}
\mu(\zeta )&=\sum_{j=1}^k \frac {1}{j} (1-\zeta )^j\bigg(\frac{\zeta }{1-\zeta } + \zeta  \sum_{j=0}^{k-2}c_j (1-\zeta )^j \bigg)\\
&=\sum_{j=0}^{k-1} \frac {1}{j+1} (1-\zeta )^j \bigg(1-(1-\zeta )-\sum_{j=0}^{k}c_{j-2} (1-\zeta )^j
+\sum_{j=0}^{k-1}c_{j-1} (1-\zeta )^j \bigg)\\
&=\sum_{j=0}^{k-1} \frac {1}{j+1} (1-\zeta )^j \sum_{j=1}^{k} \frac {1}{j} (1-\zeta )^j
-\sum_{j=2}^{k-1}\bigg(\sum_{\ell=0}^{j}\frac {1}{j-\ell+1}c_{\ell-2}\bigg) (1-\zeta )^j \\
&\quad +\sum_{j=1}^{k-1}\bigg(\sum_{\ell=0}^{j}\frac {1}{j-\ell+1}c_{\ell-1}\bigg)(1-\zeta )^j
+O\big((1-\zeta )^k\big)\\
&=1+\sum_{j=1}^{k-1} \bigg(\frac {1}{j+1} - \frac {1}{j}-\sum_{\ell=0}^{j}\frac {1}{j-\ell+1}c_{\ell-2}+\sum_{\ell=0}^{j}\frac {1}{j-\ell+1}c_{\ell-1} \bigg)(1-\zeta )^j +O\big((1-\zeta )^k\big) \\
&=1+\sum_{j=1}^{k-1} \bigg(-\frac {1}{j(j+1)}-\sum_{\ell=1}^{j-1}\frac {1}{j-\ell}c_{\ell-1}+\sum_{\ell=0}^{j-1}\frac {1}{j-\ell}c_{\ell} \bigg)(1-\zeta )^j +O\big((1-\zeta )^k\big).
\end{align*}
Thus by choosing $c_\ell$, $\ell=0,\dots,k-2$, such that
\begin{align} \label{formula_cj}
\begin{aligned}
\sum_{\ell=0}^{j-1}\frac {1}{j-\ell}c_{\ell}=
\frac {1}{j(j+1)} +\sum_{\ell=1}^{j-1}\frac {1}{j-\ell}c_{\ell-1},
\quad j=1,\dots,k-1 ,
\end{aligned}
\end{align}
Criterion \eqref{crit:mu} follows.
The coefficients $a_j^{(k)}$ can be computed recursively from \eqref{formula_cj}
and \eqref{formula_aj}, and are given in Table \ref{tab:an}.
It is worth noting that the result for $k=2$ recovers exactly the correction
in \cite{JinLazarovZhou:SISC2016}, and thus our algebraic construction generalizes
the approach in \cite{JinLazarovZhou:SISC2016}.

\begin{table}[htb!]
\caption{The coefficients $a_j^{(k)}$ computed by \eqref{formula_aj}}
\label{tab:an}
\centering
     \begin{tabular}{|c|ccccc|}
\hline
      order of BDF &  $a_1^{(k)}$  & $a_2^{(k)}$  & $a_3^{(k)}$ & $a_4^{(k)}$ & $a_5^{(k)}$   \\[2pt]
\hline
       $k=2$       & $ \frac{1}{2}$ & & & & \\
\hline
       $k=3$          &$\frac{11}{12}$  & $-\frac{5}{12}$    &   &  &  \\[2pt]
 \hline
       $k=4$          &$\frac{31}{24}$  & $-\frac{7}{6}$  & $\frac{3}{8}$ &  &  \\[2pt]
\hline
       $k=5$          &$\frac{1181}{720}$  & $-\frac{177}{80}$   & $\frac{341}{240}$ & $-\frac{251}{720}$  & \\[2pt]
\hline
       $k=6$          &$\frac{2837}{1440}$& $-\frac{2543}{720}$   &$\frac{17}{5}$ & $-\frac{1201}{720}$ & $\frac{95}{288}$  \\[2pt]
\hline
     \end{tabular}
\end{table}

Next we compute the coefficients $b_{\ell,j}^{(k)}$. First we expand  $\frac{\gamma_\ell(\zeta )}{\ell !}-
\frac{1}{\delta(\zeta )^{\ell+1}}$ in $1-\zeta $ as
\begin{align}\label{Taylor-gln}
\frac{\gamma_\ell(\zeta )}{\ell !}-\frac{1}{\delta(\zeta )^{\ell+1}}
=\sum_{j=0}^{k-\ell-2} g_{\ell,j}^{(k)}(1-\zeta )^j + O(|1-\zeta |^{k-\ell-1}),
\end{align}
and then choose the coefficients $b_{\ell,j}^{(k)}$, $j=1,\dots,k-1$ to satisfy \eqref{crit:b}.
To this end, we rewrite $\sum_{j=1}^{k-1} b_{\ell,j}^{(k)} \zeta ^j$ into the following form:
\begin{align}\label{compute-bln}
\sum_{j=1}^{k-1} b_{\ell,j}^{(k)} \zeta ^j=\zeta \sum_{j=0}^{k-2} d_{\ell,j}^{(k)}(1- \zeta )^j
=\sum_{j=0}^{k-2} d_{\ell,j}^{(k)}(1- \zeta )^j-\sum_{j=1}^{k-1} d_{\ell,j-1}^{(k)}(1- \zeta )^j.
\end{align}
Then it suffices to choose
\begin{subequations}\label{compute-dln}
\begin{align}
&d_{\ell,0}^{(k)}=-g_{\ell,0}^{(k)},\\
&d_{\ell,j}^{(k)}=d_{\ell,j-1}^{(k)}-g_{\ell,j}^{(k)} &&\mbox{for}\,\,\, j=1,\dots,k-\ell-2,\\
&d_{\ell,j}^{(k)}=0 &&\mbox{for}\,\,\, j=k-\ell-1,\dots,k-2 .
\end{align}
\end{subequations}
Now the coefficients $b_{\ell,j}^{(k)}$ can be computed recursively using
\eqref{Taylor-gln}, \eqref{compute-dln} and \eqref{compute-bln}, and
the results are given in Table \ref{tab:bln}.
Note that for $k=4$ and $6$, the coefficients $b_{k-2,j},j=1,2\ldots,k-1$ vanish identically.

\begin{table}[htb!]
\caption{The coefficients $b_{\ell,j}^{(k)}$.}
\label{tab:bln}
\centering
     \begin{tabular}{|c c|ccccc|}
\hline
      order of BDF &   & $b_{\ell,1}^{(k)}$  & $b_{\ell,2}^{(k)}$  & $b_{\ell,3}^{(k)}$ & $b_{\ell,4}^{(k)}$ & $b_{\ell,5}^{(k)}$     \\[2pt]
\hline
       $k=3$          &$\ell=1$   &$\frac{1}{12}$  &  0   &  &    &  \\[2pt]
\hline
       $k=4$          &$\ell=1$   &$\frac{1}{6}$  & $-\frac{1}{12}$    & $0$   &  &  \\[3pt]
                           &$\ell=2$   & $0$                &  $0$  & $0$   & & \\[2pt]
\hline
       $k=5$          &$\ell=1$   &$\frac{59}{240}$  & $-\frac{29}{120}$   & $\frac{19}{240}$ & $0$ &  \\[3pt]
                           &$\ell=2$   & $\frac{1}{240}$                 & $-\frac{1}{240}$  & $0$ & $0$ &  \\[3pt]
                           &$\ell=3$   &$\frac{1}{720}$ & $0$  & $0$ & $0$ & \\[2pt]
\hline
       $k=6$          &$\ell=1$   &$\frac{77}{240}$& $-\frac{7}{15}$  &$\frac{73}{240}$ & $-\frac{3}{40}$  & 0\\[3pt]
                           &$\ell=2$   & $\frac{1}{96}$             & $-\frac{1}{60}$  & $\frac{1}{160}$  & $0$ & 0 \\[3pt]
                           &$\ell=3$   & $-\frac{1}{360}$ &  $\frac{1}{720}$ & $0$ & $0$  & 0\\[3pt]
                           &$\ell=4$   & $0$ & $0$ & $0$ & $0$  & 0\\[2pt]
\hline
     \end{tabular}
\end{table}

\subsection{Error estimates}
Last we state the error estimate for \eqref{eqn:BDF-CQ}. The proof relies on the splitting $u(t_n)-U^n
=w(t_n)-W^n$ and the representations \eqref{eqn:semisol} and \eqref{eqn:Rep-Wh}, and then bounding each
term using \eqref{eqn:resol}. The details can be found in Appendix \ref{app:conv-sub}.
\begin{theorem}\label{thm:conv}
Let Criteria \eqref{crit:mu} and \eqref{crit:b} hold. Then
for the solution $U^n$ to the corrected scheme \eqref{eqn:BDF-CQ}, the following error estimate holds for any $t_n>0$
\begin{equation*}
  \begin{aligned}
  \|U^n-u(t_n)\|_{L^2(\Omega)}  \leq & c\tau^k \bigg(t_n^{ \alpha -k }  \|f(0)+Av\|_{L^2(\Omega)} + \sum_{\ell=1}^{k-1} t_n^{ \alpha+\ell -k }  \|\partial_t^{\ell}f(0)\|_{L^2(\Omega)}\\
   &+\int_0^{t_n}(t_n-s)^{\alpha-1}\|\partial_s^{k}f(s)\|_{L^2(\Omega)}ds\bigg).
  \end{aligned}
\end{equation*}
\end{theorem}

\begin{remark}
Theorem \ref{thm:conv} implies that for any fixed $t_n>0$, the convergence rate is $O(\tau^k)$ for BDF$k$.
In order to have a \emph{uniform} rate $O(\tau^k)$, the following compatibility conditions are needed:
\begin{equation*}
    f(0) + Av = 0, \quad \mbox{and}\quad     \partial_t^{(\ell)} f(0) = 0, \quad \ell=1,\ldots,k-1,
\end{equation*}
concurring with known results on convolution quadrature \cite{Lubich:2004}. In the absence of these conditions, the error estimate
deteriorates as $t\to0$, which is consistent with the corresponding regularity theory: the solution {\rm(}and its
derivatives{\rm)} exhibits weak singularity at $t=0$  \cite{SakamotoYamamoto:2011}.
\end{remark}

\begin{remark}\label{rmk:error-nonsmooth}
The error estimate in Theorem \ref{thm:conv} requires $Av\in L^2(\Omega)$, i.e., the initial data
$v$ is reasonably smooth. Upon minor modifications of the proof in Appendix \ref{app:conv-sub}, one can derive a similar error
estimate for $v\in L^2(\Omega)$:
\begin{equation*}
  \|U^n-u(t_n)\|_{L^2(\Omega)}  \leq c\tau^k \bigg( t_n^{-k} \|v\|_{L^2(\Omega)} + \sum_{\ell=0}^{k-1} t_n^{ \alpha+\ell -k }  \|\partial_t^{\ell}f(0)\|_{L^2(\Omega)}+\int_0^{t_n}(t_n-s)^{\alpha-1}\|\partial_s^{k}f(s)\|_{L^2(\Omega)}ds\bigg).
\end{equation*}
\end{remark}

\section{Corrected BDF for diffusion-wave problem}\label{sec:diff-wave}

Now we extend the strategy in Section \ref{sec:scheme} to the diffusion-wave problem, i.e., $1<\alpha<2$:
\begin{equation*}
  \partial_t^\alpha (u(t)-v-tb) -Au(t) = f(t),
\end{equation*}
with the initial conditions $u(0)=v$ and $u^\prime(0)=b$, where
\begin{equation}
   \partial_t^\alpha u(t):= \frac{1}{\Gamma(2-\al)}\frac{d^2}{dt^2} \int_0^t(t-s)^{1-\al}u(s)\, ds .
\end{equation}
The main differences from the subdiffusion case lie in
the extra initial condition $b$ and better temporal smoothing property \cite{JinLazarovZhou:SISC2016}. A straightforward
implementation of BDF$k$ can fail to yield the  $O(\tau^k)$ rate, as the subdiffusion case, and further requires unnecessarily
high regularity on $f$. We shall develop a corrected scheme to take care of both issues. First, in order to fully exploit
the extra smoothing, we rewrite the {source term} $f$ as $f=\partial_tg$ with $g=\partial_t^{-1}f$.
Then the diffusion-wave equation can be rewritten as
\begin{equation}\label{eqn:diff-wave}
  \partial_t^\alpha (u-v-tb) - Au = \partial_t g,
\end{equation}
Next we correct the starting $k-1$ steps, and seek approximations
$U^n$, $n=1,\dots,N$, by
\begin{equation}\label{eqn:BDF-CQ-dw}
\begin{aligned}
&\bPtau^\alpha (U-v-tb)^n  - A U^n = a_n^{(k)} A v + c_{n}^{(k)}\tau  Ab + \bar\partial_\tau g^n +
\sum_{\ell=1}^{k-2} b_{\ell,n}^{(k)}\tau^{\ell-1} \partial_t^{\ell-1} f(0) ,
&& 1\le n\le k-1,\\
&\bPtau^\alpha (U-v-tb)^n  -  A U^n = \bar \partial_\tau g^n , &&k\le n\le N.
\end{aligned}
\end{equation}
The scheme involves $\bar\partial_\tau g^n$, instead of $f^n$,
which enables one to relax the regularity requirement on $f$.
The correction terms are to ensure the desired $O(\tau^k)$ rate.

Now we derive the criterion for choosing the coefficients in \eqref{eqn:BDF-CQ-dw} using Laplace
transform and generating function. First, since $g(0)=0$, $g(t)$ can be split into
\begin{equation}\label{eqn:decomp-rhs-dw}
   g(t)= \sum_{\ell=1}^{k-2} \frac{t^{\ell} }{\ell !}\partial_t^{\ell}g(0) + R_{k}=\sum_{\ell=1}^{k-2} \frac{t^{\ell} }{\ell !}\partial_t^{\ell-1}f(0) +R_{k},
\end{equation}
where $R_k$ is the local truncation error
$R_{k} =\frac{t^{k-1}}{(k-1)!}\partial_t^{k-1}g(0) +\frac{t^{k-1}}{(k-1)!}\ast\partial_t^{k}g(t).$
With the splitting \eqref{eqn:decomp-rhs-dw}, the function $w=u-v-tb$ satisfies
\begin{equation*}
  \partial_t^\alpha w - Aw = Av +tAb + \sum_{\ell=1}^{k-2} \partial_t \frac{t^{\ell} }{\ell !}\partial_t^{\ell-1}f(0)+ \partial_t R_{k}.
\end{equation*}
Then by Laplace transform, we derive a representation of the continuous solution $w(t)$:
\begin{equation}\label{eqn:sol-rep-dw-cont}
   w(t) =  \frac{1}{2\pi \mathrm{i}}\int_{\Gamma_{\theta,\delta}}  e^{zt}K(z) (Av + z^{-1}Ab) dz
  +\frac{1}{2\pi \mathrm{i}}\int_{\Gamma_{\theta,\delta}}
  e^{zt} zK(z) \bigg(\sum_{\ell=1}^{k-2} \frac{1}{z^{\ell}} \partial_t^{\ell-1}f(0)
    + z \widehat R_{k}(z)\bigg)dz, \vspace{-15pt}
\end{equation}
where the angle $\theta \in (\pi/2,\pi)$ is sufficiently close to $\pi/2$ such that
$\alpha\theta<\pi$, and $\delta$ is small.

Since BDF$k$ is $A(\vartheta_k)$-stable, the scheme \eqref{eqn:BDF-CQ-dw} is unconditionally
stable for any $\alpha <\alpha^*(k):= \pi/(\pi-\vartheta_k)$. The critical value $\alpha^*(k)$ is $1.91$,
$1.68$, $1.40$ and $1.11$ for $k=3,\ldots,6$. In contrast, for $\alpha\geq \alpha^*(k)$, it is
only conditionally stable. Note that for any $\alpha\in(1,2)$,
the curve $\delta(e^{-\mathrm{i}\theta})^\alpha$ is not tangent to the real axis at the origin (i.e., $\theta$
close to zero). This naturally gives rise to the following condition.
\begin{condition}\label{assump:CFL}
Let $r(A)$ be the numerical radius of $A$, and the following condition holds:
{\rm(i)} The fractional order $\alpha<\alpha^*(k)$ or
{\rm(ii)} The fractional order $\alpha\ge \alpha^*(k)$ and $\tau^\alpha r(A) \leq c(\alpha,k)-\gamma$ for some $\gamma>0$,
where the constant $c(\alpha,k)$ is given by the intersection point of
$\{\delta(\zeta)^\alpha: |\zeta|=1\}$ with the negative real axis {\rm(}closest to the origin{\rm)}.
\end{condition}

\begin{remark}\label{rmk:stable}
Condition \ref{assump:CFL}(ii) specifies the CFL condition on the time step size $\tau$ {\rm(}so it
holds only if $r(A)<\infty${\rm)}. The CFL constant $c(\alpha,k)$ is not available in closed form,
but can be determined numerically; see Fig. \ref{fig:cfl} for the values.
\end{remark}

It is interesting to observe the qualitative differences of BDFs of different order. For example, the CFL constant $c(\alpha,6)$
of BDF6 does not approach zero even for $\alpha$ tends to $2$; and there is an interval of $\alpha$
values for which the CFL constant $c(\alpha,4)$ for BDF4 is larger than $c(\alpha,3)$ for BDF3, i.e., BDF4 is less stringent in
time step size.

\begin{figure}[hbt!]
  \centering
  \includegraphics[width=0.48\textwidth]{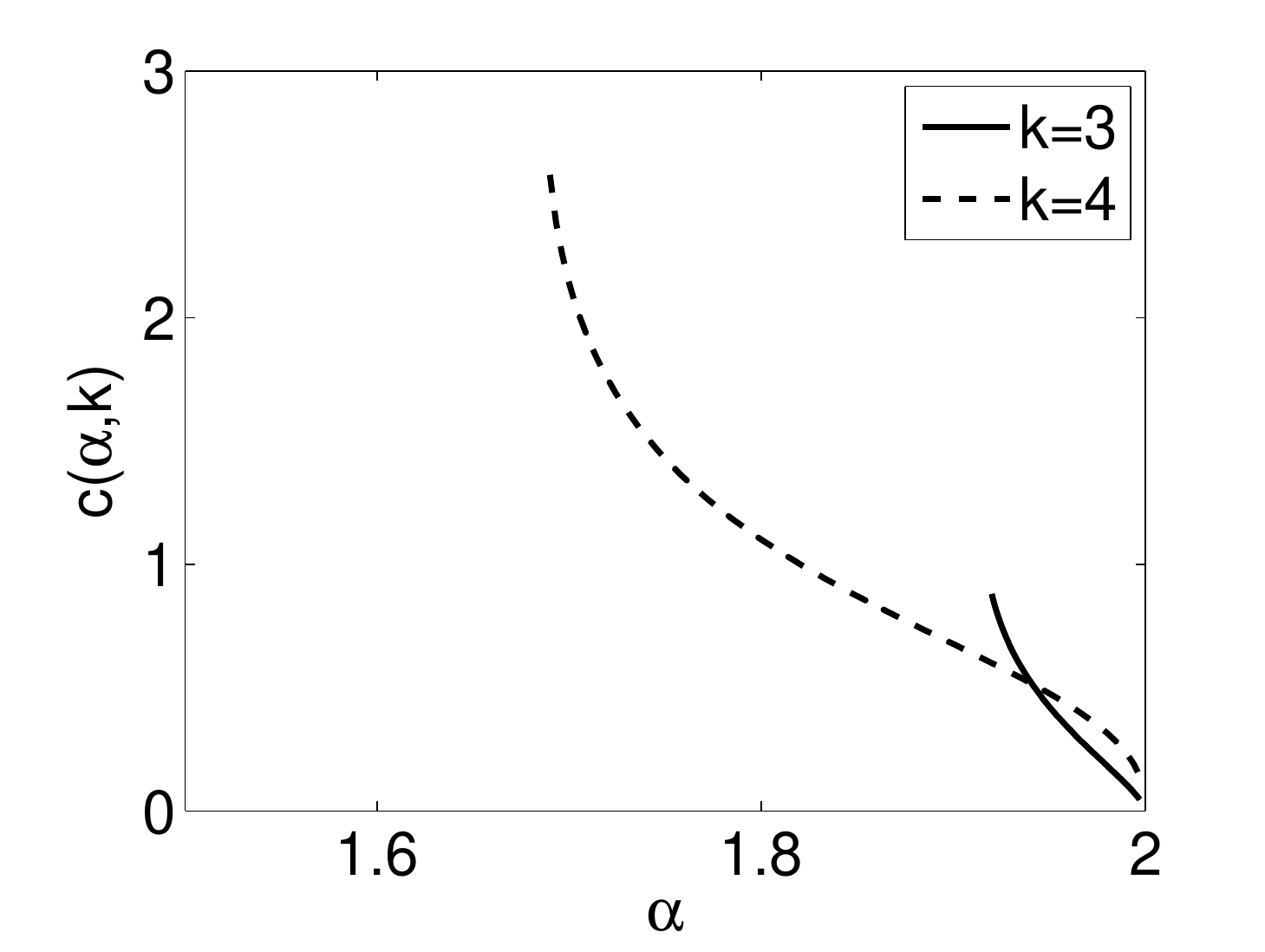} \includegraphics[width=0.48\textwidth]{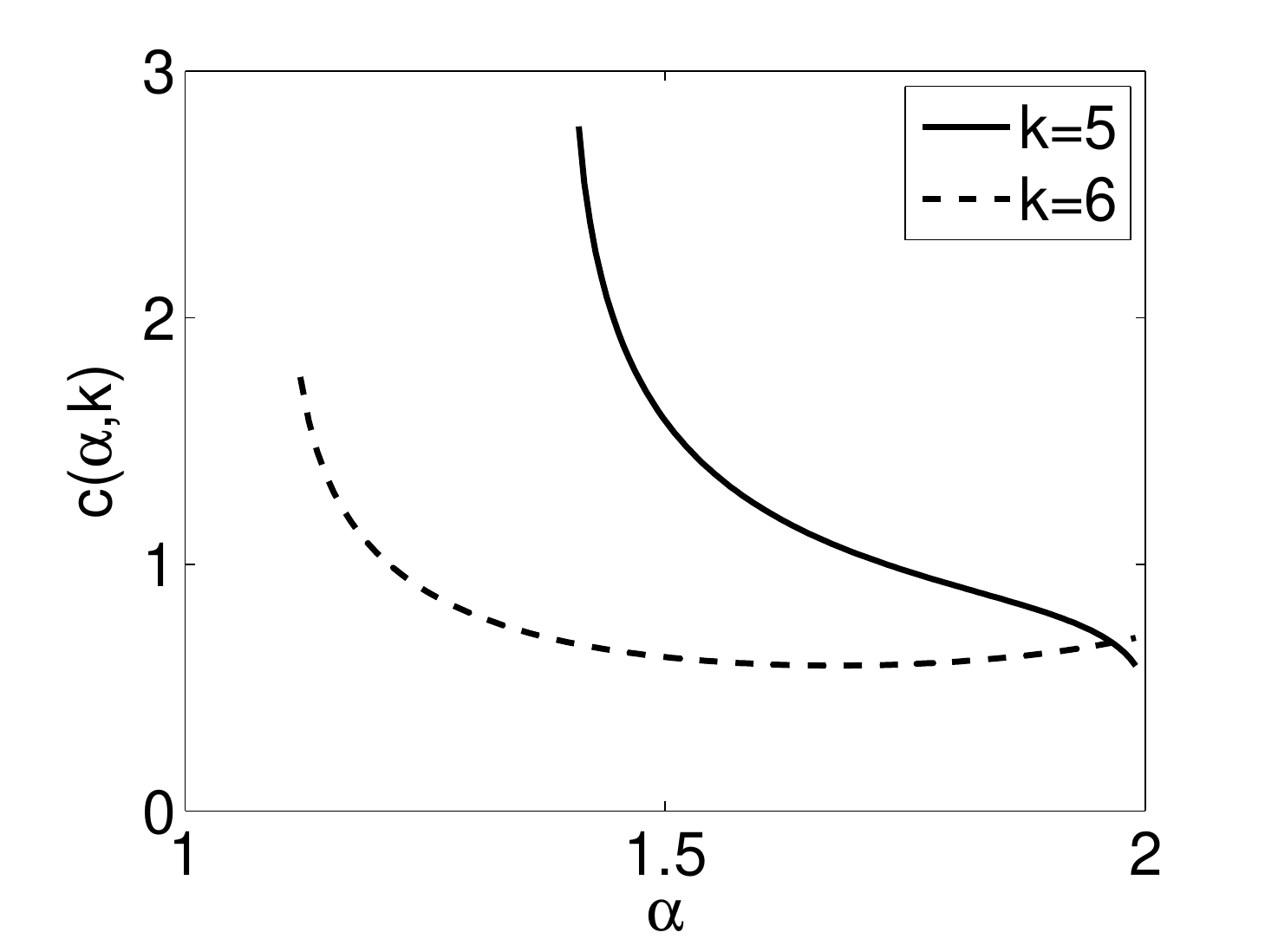}
  \caption{The CFL constant $c(\alpha,k)$ for BDF$k$, $k=3,4,5,6$, at different $\alpha$ values.}\label{fig:cfl}
\end{figure}

The next result gives the representation of the solution $W^n=U^n-v-t_nb$,
which follows from simple yet lengthy computations, cf. Appendix \ref{app:sol-rep-dw}.
\begin{theorem}\label{thm:solurep-dw}
Under Condition \ref{assump:CFL},
the discrete solution $W^n:=U^n-v-t_nb$ is given by
\begin{align}\label{eqn:sol-rep-dw-dis}
W^n=& \frac{1}{2\pi\mathrm{i} }\int_{\Gamma^\tau_{\theta,\delta}}e^{zt_{n}} \mu(e^{-z\tau}) K( \delta_\tau(e^{-z\tau}))A v\,dz \nonumber  \\
&+\frac{1}{2\pi\mathrm{i}}\int_{\Gamma^\tau_{\theta,\delta}}e^{zt_n}K(\delta_\tau(e^{-z\tau}))\delta_\tau(e^{-z\tau}) \bigg(\gamma_1(e^{-z\tau})+\sum_{j=1}^{k-1}c_j^{(k)}e^{-zt_j}\bigg)\tau^2 Ab\,dz \nonumber \\
  &+\frac{1}{2\pi\mathrm{i} }\int_{\Gamma^\tau_{\theta,\delta}}e^{zt_{n}}
  \delta_\tau(e^{-z\tau})K(\delta_\tau(e^{-z\tau}))  \sum_{\ell=1}^{k-2}\bigg(\delta(e^{-z\tau})\frac{\gamma_\ell(e^{-z\tau})}{\ell !}+\sum_{j=1}^{k-1} b_{\ell,j}^{(k)}e^{-zt_j}\bigg)\tau^{\ell}\partial_t^{\ell-1}f(0)\,dz\nonumber \\[-5pt]
&
+\frac{1}{2\pi\mathrm{i} }\int_{\Gamma^\tau_{\theta,\delta}}e^{zt_{n}}
  \delta_\tau(e^{-z\tau})^2K(\delta_\tau(e^{-z\tau}))\tau \widetilde  R_{k}(e^{-z\tau})
\,dz ,
\end{align}
with the contour
$\Gamma_{\theta,\delta}^\tau :=\{ z\in \Gamma_{\theta,\delta}:|\Im(z)|\le {\pi}/{\tau} \}$ {\rm(}oriented with an increasing imaginary part{\rm)},
for some $\theta$ sufficiently close to $\pi/2$, where $\mu(\zeta ) $ and
$\gamma_\ell(\zeta )$ are defined in \eqref{eqn:gamma-mu}.
\end{theorem}

Proceeding like before, from the solution representations \eqref{eqn:sol-rep-dw-cont} and \eqref{eqn:sol-rep-dw-dis},
we deduce the following algebraic criteria for choosing the coefficients $a_j^{(k)}$, $c_j^{(k)}$ and $b_{\ell,n}^{(k)}$:
\begin{align}
  |\mu(\zeta)-1| & \leq c|1-\zeta|^k,\label{crit:mu-dw}\\
  \bigg|\gamma_1(\zeta)+\sum_{j=1}^{k-1}c_j^{(k)}\zeta^j-\frac{1}{\delta(\zeta)^2}\bigg|&\leq c|1-\zeta|^{k-2}, \label{crit:c-dw}\\
  \bigg|\delta(\zeta)\frac{\gamma_\ell(\zeta)}{\ell !}+\sum_{j=1}^{k-1} b_{\ell,j}^{(k)}\zeta^j-\frac{1}{\delta(\zeta)^\ell}\bigg|&\leq c|1-\zeta|^{k-\ell},\quad \ell=1,2,\ldots,k-2, \label{crit:b-dw}
\end{align}
where the functions $\mu(\zeta)$ and $\gamma_\ell(\zeta)$ are defined in \eqref{eqn:gamma-mu}.

By comparing Criterion \eqref{crit:mu-dw} with \eqref{crit:mu}, and respectively Criterion \eqref{crit:c-dw}
with \eqref{crit:b}, the coefficients $a_j^{(k)}$ are identical with that for $\alpha\in(0,1)$, and respectively $c_j^{(k)}$
with $b_{1,j}^{(k)}$ for $\alpha\in(0,1)$. However, due to the presence of the extra factor $\delta(\zeta)$,
the coefficients $b_{\ell,j}^{(k)}$ are different from that of the case $0<\alpha<1$, and have to be determined.
The procedure for computing $b_{\ell,j}^{(k)}$
is similar to that in Section \ref{ssec:coef-sub}, and the results are given in Table \ref{tab:bln-dw}.

\begin{table}[htb!]
\caption{The coefficients $b_{\ell,j}^{(k)}$ according to Criterion \eqref{crit:b-dw}.}
\label{tab:bln-dw}
\centering
     \begin{tabular}{|c c|ccccc|}
\hline
      order of BDF &   & $b_{\ell,1}^{(k)}$  & $b_{\ell,2}^{(k)}$  & $b_{\ell,3}^{(k)}$ & $b_{\ell,4}^{(k)}$ & $b_{\ell,5}^{(k)}$     \\[2pt]
\hline
       $k=3$          &$\ell=1$   &$\frac{1}{12}$  &  $-\frac{1}{12}$   &  &    &  \\[2pt]
\hline
       $k=4$          &$\ell=1$   &$\frac{5}{24}$  & $-\frac{1}{3}$    & $\frac{1}{8}$   &  &  \\[3pt]
                           &$\ell=2$   & $0$                &  $0$  & $0$   & & \\[2pt]
\hline
       $k=5$          &$\ell=1$   &$\frac{257}{720}$  & $-\frac{187}{240}$   & $\frac{137}{240}$ & $-\frac{107}{240}$ &  \\[3pt]
                           &$\ell=2$   & $\frac{1}{240}$                 & $-\frac{1}{120}$  & $\frac{1}{240}$    & $0$ &  \\[3pt]
                           &$\ell=3$   &$-\frac{1}{720}$ & $\frac{1}{720}$  & $0$ & $0$ & \\[2pt]
\hline
       $k=6$          &$\ell=1$   &$\frac{749}{1440}$& $-\frac{1031}{720}$  &$\frac{31}{20}$ & $-\frac{577}{720}$  & $\frac{47}{288}$\\[3pt]
                           &$\ell=2$   & $\frac{1}{80}$             & $-\frac{1}{30}$  & $\frac{7}{240}$  & $-\frac{1}{120}$ & 0 \\[3pt]
                           &$\ell=3$   & $-\frac{1}{288}$ &  $\frac{1}{180}$ & $-\frac{1}{480}$ & $0$  & 0\\[3pt]
                           &$\ell=4$   & $0$ & $0$ & $0$ & $0$  & 0\\[2pt]
\hline
     \end{tabular}
\end{table}

Last, we state the error estimate for the approximation $U^n$. The
proof is similar to that of Theorem \ref{thm:conv-dw}, but with $g=\partial_t^{-1}f$
in place of $f$. It is briefly sketched in Appendix \ref{app:diff-wave}.
\begin{theorem}\label{thm:conv-dw}
Let Criteria \eqref{crit:mu-dw}, \eqref{crit:c-dw} and \eqref{crit:b-dw} hold, and Condition \ref{assump:CFL} be
fulfilled. Then for the solution $U^n$ to \eqref{eqn:BDF-CQ-dw}, the following error estimate holds {for any $t_n>0$}
\begin{equation*}
\begin{aligned}
  \|U^n-u(t_n)\|_{L^2(\Omega)}  \leq  & c\tau^k \bigg(t_n^{ \alpha -k }  \|f(0)+Av\|_{L^2(\Omega)} + t_n^{ \alpha+1 -k }\| A b \|_{L^2(\Omega)} \\
 &\quad + \sum_{\ell=1}^{k-2} t_n^{ \alpha+\ell -k }\|\partial_t^{\ell}f(0)\|_{L^2(\Omega)}+\int_0^{t_n}(t_n-s)^{\alpha-2}\|\partial_s^{k-1}f(s)\|_{L^2(\Omega)}ds\bigg).
\end{aligned}
\end{equation*}
\end{theorem}
\begin{remark}
Theorem \ref{thm:conv-dw} only requires $(k-1)^{\rm th}$ order derivative of $f$ in time, instead of $k^{\rm th}$ order derivative of $f$ as in Theorem \ref{thm:conv}. Thus it indeed relaxes the regularity
condition.
\end{remark}

\section{Numerical experiments and discussions}\label{sec:numerics}
Now we present numerical results to show the efficiency and accuracy of the schemes
\eqref{eqn:BDF-CQ} and \eqref{eqn:BDF-CQ-dw} in one-spatial dimension, on the unit interval $\Omega=(0,1)$. In space, it is
discretized with the piecewise linear Galerkin finite element method \cite{JinLazarovZhou:SIAM2013}: we
divide $\Omega$ into $M$ equally spaced subintervals with a mesh size
$h=1/M$. Since the convergence behavior of the spatial discretization is well understood,
we focus on the temporal convergence. In the computation, we fix the time step size
$\tau$ at $\tau=t/N$, where $t$ is the time of interest. We measure the accuracy
by the normalized errors $e^N=\|u(t_N)-U^N \|_{L^2(\Omega)}/ \| u(t_N) \|_{L^2(\Omega)}$, where the reference solution
$u(t_N)$ is computed using a much finer mesh. All the computations are carried out in MATLAB {R2015a}
on a personal laptop, and further, in order to observe error beyond
double precision, we employ the {Multiprecision Computing Toolbox\footnote{\url{http://www.advanpix.com/}, last accessed on January 11, 2017.} for MATLAB}.

\subsection{Numerical results for subdiffusion}

In the subdiffusion case, we consider the following two examples:
\begin{itemize}
  \item[(a)]  $v=x(1-x) \in H^2(\Omega)\cap H_0^1(\Omega)$ and $f\equiv 0$;
  \item[(b)] $v\equiv0$ and $f(x,t)=\cos(t)(1+\chi_{(0,1/2)}(x))$.
\end{itemize}

The numerical results for case (a) by the corrected scheme \eqref{eqn:BDF-CQ} are presented in Table
\ref{tab:v-correct-smooth}, where the numbers in the bracket denote the theoretical rate predicted
by Theorem \ref{thm:conv}. It converges steadily at an $O(\tau^k)$ rate for all
BDFs, which agrees well with the theory, showing clearly its robustness.
Surprisingly, the asymptotic convergence of BDF6 kicks in only at a relatively small time step size,
at $N=50$, which contrasts sharply with other BDF schemes. Thus in the preasymptotic regime, BDF5
is preferred over BDF6. To further illustrate Theorem
\ref{thm:conv}, in Fig. \ref{fig:error}, we plot the numerical solution by BDF5 and its error
profile. The solution decays first rapidly and then slowly, resulting in an
initial layer. This layer shows clearly the limited temporal regularity of the solution at $0$ and as a
result, the approximation error near $0$ is predominant, partly confirming the prefactor
$t_n^{\alpha-k}$ in Theorem \ref{thm:conv}.

{\small
\begin{table}[htb!]
\caption{The $L^2$-norm error $e^N$ for case (a) at $t_N=1$, by the corrected scheme
\eqref{eqn:BDF-CQ} with $h=1/100$.}
\label{tab:v-correct-smooth}
\centering
     \begin{tabular}{|c|c|lllll|c|}
     \hline
      $\alpha$ &  $k\backslash N$  &$50$   &$100$ &$200$ &$400$ & $800$   &rate \\
     \hline
             & 2  &5.66e-5 &1.39e-5 &3.46e-6 &8.64e-7 &2.16e-7 &$\approx$ 2.00 (2.00)\\
             & 3  &2.29e-6 &2.76e-7 &3.39e-8 &4.20e-9 &5.23e-10 &$\approx$ 3.01 (3.00)\\
    $ 0.25$  & 4  &1.42e-7 &8.33e-9 &5.04e-10 &3.10e-11 &1.91e-12  &$\approx$ 4.02 (4.00)\\
             & 5  &1.26e-8 &3.41e-10 &1.01e-11 &3.07e-13 &9.45e-15 &$\approx$ 5.03 (5.00)\\
             & 6  &1.09e-5 &1.60e-9 &2.55e-13 &3.82e-15 &5.83e-17 &$\approx$ 6.04 (6.00)\\
      \hline
             & 2  &1.74e-4 &4.30e-5 &1.07e-5 &2.65e-6 &6.62e-7 &$\approx$ 2.00 (2.00)\\
             & 3  &7.73e-6 &9.29e-7 &1.14e-7 &1.41e-8 &1.76e-9 &$\approx$ 3.01 (3.00)\\
    $ 0.5$   & 4  &5.12e-7 &2.98e-8 &1.80e-9 &1.10e-10 &6.83e-12 &$\approx$ 4.02 (4.00)\\
             & 5  &4.75e-8 &1.27e-9 &3.76e-11 &1.14e-12 &3.52e-14 &$\approx$ 5.03 (5.00)\\
             & 6  &3.01e-5 &2.79e-9 &9.85e-13 &1.47e-14 &2.25e-16 &$\approx$ 6.05 (6.00)\\
      \hline
             & 2  &4.84e-4 &1.19e-4 &2.93e-5 &7.30e-6 &1.82e-6 &$\approx$ 2.00 (2.00)\\
             & 3  &2.55e-5 &3.04e-6 &3.72e-7 &4.60e-8 &5.71e-9 &$\approx$ 3.01 (3.00)\\
    $ 0.75$  & 4  &1.94e-6 &1.11e-7 &6.68e-9 &4.09e-10 &2.53e-11 &$\approx$ 4.02 (4.00)\\
             & 5  &2.95e-7 &5.30e-9 &1.55e-10 &4.70e-12 &1.45e-13 &$\approx$ 5.03 (5.00)\\
             & 6  &1.67e-3 &3.01e-7 &4.53e-12 &6.61e-14 &1.01e-15 &$\approx$ 6.07 (6.00)\\
      \hline
     \end{tabular}
\end{table}}

To illustrate the impact of initial correction, we present in Table \ref{tab:v-nocorrect} the numerical
results by the uncorrected BDF scheme \eqref{eqn:BDF-CQ-0}, and two popular finite difference schemes,
i.e., L1 scheme \cite{LinXu:2007} and L1-2 scheme \cite{GaoSunZhang:2014,LvXu:2016}. The uncorrected
BDF$k$ scheme can only achieve a first-order convergence, and all BDF schemes have almost identical accuracy,
irrespective of the order $k$. This low-order convergence is due to the poor approximation at
the initial steps, which persists in the numerical solutions at later steps. Meanwhile, for sufficiently
smooth solutions, the L1 and L1-2 schemes converge at a rate $O(\tau^{2-\alpha})$ and $O(\tau^{3-\alpha})$,
respectively. For general problem data, the L1 scheme converges at an $O(\tau)$ rate \cite{JinLazarovZhou:2016ima}.
The L1 and L1-2 schemes can only deliver an empirical $O(\tau)$ rate for case (a),
due to insufficient solution regularity for general problem data. Although not presented, it is noted
that the numerical results for other fractional orders are similarly. Therefore, the correction
is necessary in order to retain the desired rate, even for smooth initial data.

\begin{figure}[hbt!]
\vspace{-15pt}
\centering
\subfigure[numerical solution]{
\includegraphics[trim = .1cm .1cm .1cm .1cm, clip=true,width=6.5cm]{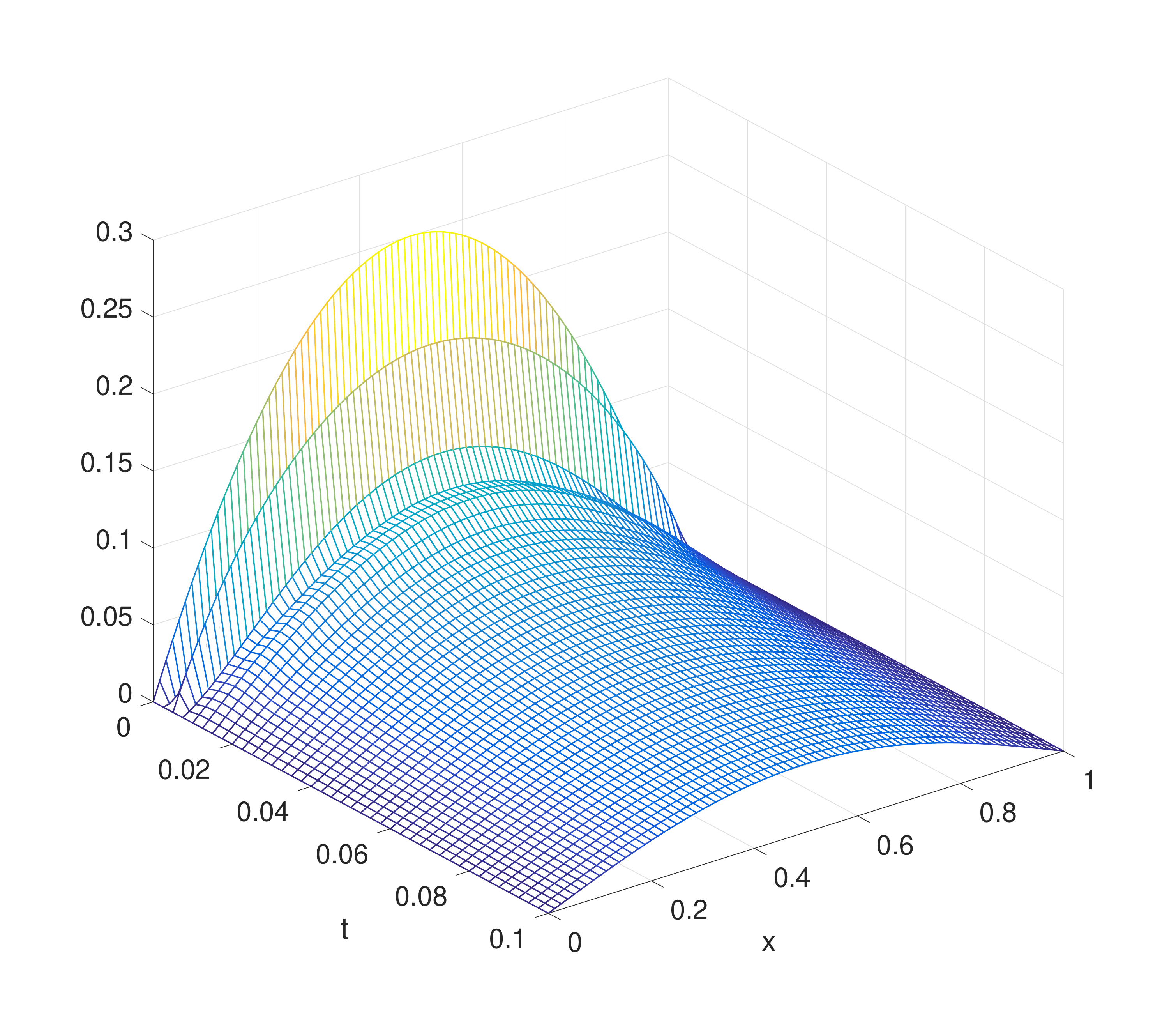}}
\subfigure[error profile]{
\includegraphics[trim = .1cm .1cm .1cm .1cm, clip=true,width=6.5cm]{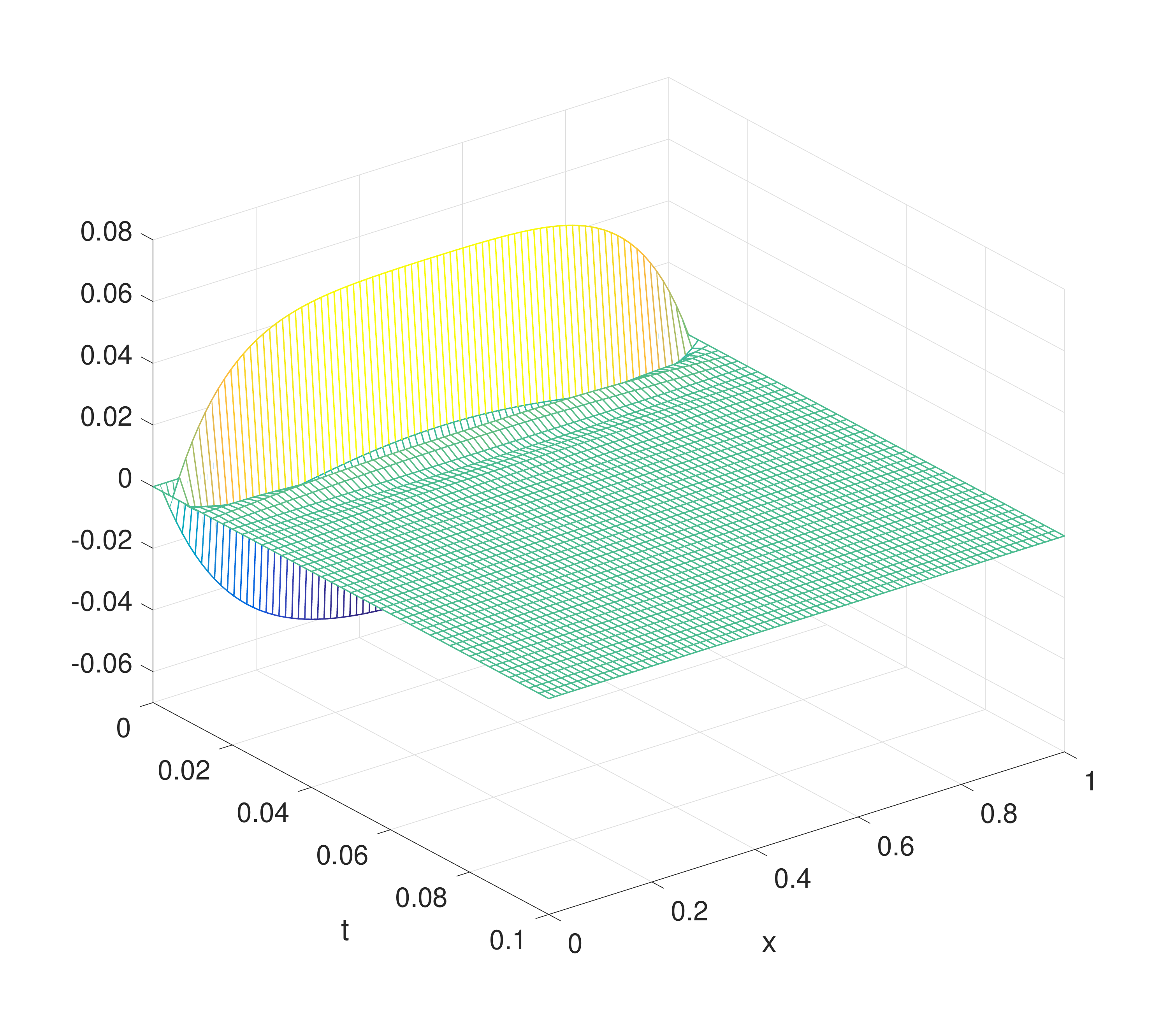}}
 \caption{Numerical solution and error profile for case (a), with $\alpha=0.5$, $h=1/100$, $\tau=0.002$ and BDF5.\label{fig:error} }
\end{figure}

{\small
\begin{table}[htb!]
\caption{The $L^2$-norm error $e^N$ for case (a)
at $t_N=1$, by the uncorrected scheme \eqref{eqn:BDF-CQ-0} with $h=1/100$.}
\label{tab:v-nocorrect}
\centering
     \begin{tabular}{|c|c|ccccc|c|}
     \hline
      $\alpha$ &  $N$  &$50$ &$100$ &$200$ & $400$ & $800$  &rate \\
      \hline
             & BDF3  &4.98e-3 &2.48e-3 &1.24e-3 &6.19e-4 &3.09e-4 &$\approx$ 1.00 ($1.00$)\\
             & BDF4  &4.97e-3 &2.48e-3 &1.24e-3 &6.19e-4 &3.09e-4 &$\approx$ 1.00 ($1.00$)\\
     $ 0.5$  & BDF5  &4.97e-3 &2.48e-3 &1.24e-3 &6.19e-4 &3.09e-4 &$\approx$ 1.00 ($1.00$)\\
             & BDF6  &4.94e-3 &2.48e-3 &1.24e-3 &6.19e-4 &3.09e-4 &$\approx$ 1.00 ($1.00$)\\
             & L1     &5.10e-3 &2.52e-3 &1.25e-3 &6.24e-4 &3.11e-4 &$\approx$ 1.04 ($1.00$)\\
             & L1-2   &3.57e-3 &1.73e-3 &8.40e-4 &4.08e-4 &1.99e-4 &$\approx$ 1.04 ($--$)\\
      \hline
     \end{tabular}
\end{table}}

Next we consider the inhomogeneous problem in case (b). Since the source term $f$ is smooth in time,
Theorem \ref{thm:conv} is applicable, which predicts an $O(\tau^k)$ rate for the corrected BDF$k$
scheme \eqref{eqn:BDF-CQ}. This is fully supported by the numerical results in Table \ref{tab:f-correct}.
Like before, the uncorrected scheme \eqref{eqn:BDF-CQ-0} and the L1 and L1-2 schemes can only achieve
an $O(\tau)$ rate, despite the smoothness of the problem data, cf. Table \ref{tab:f-nocorrect}.

{\small
\begin{table}[htb!]
\caption{The $L^2$-norm error $e^N$ for case (b) at $t_N=1$, by the corrected
scheme \eqref{eqn:BDF-CQ} with $h=1/100$.}
\label{tab:f-correct}
\centering
     \begin{tabular}{|c|c|lllll|c|}
     \hline
      $\alpha$ &  $k\backslash N$  &$50$ &$100$ &$200$ & $400$ & $800$  &rate \\
     \hline
             & 2  &6.67e-6 &1.65e-6 &4.10e-7 &1.02e-7 &2.55e-8 &$\approx$ 2.00 (2.00)\\
             & 3  &2.68e-7 &3.20e-8 &3.91e-9 &4.83e-10 &6.00e-11 &$\approx$ 3.01 (3.00)\\
    $ 0.25$  & 4  &2.14e-8 &1.25e-9 &7.57e-11 &4.65e-12 &2.88e-13 &$\approx$ 4.02 (4.00)\\
             & 5  &1.90e-9 &5.11e-11 &1.51e-12 &4.61e-14 &1.42e-15 &$\approx$ 5.03 (5.00)\\
             & 6  &1.63e-6 &2.40e-10 &3.79e-14 &5.68e-16 &8.67e-18 &$\approx$ 6.05 (6.00)\\
      \hline
             & 2  &1.76e-5 &4.35e-6 &1.08e-6 &2.70e-7 &6.62e-8 &$\approx$ 2.00 (2.00)\\
             & 3  &6.35e-7 &7.56e-8 &9.22e-9 &1.14e-9 &1.42e-10 &$\approx$ 3.01 (3.00)\\
    $ 0.5$   & 4  &5.23e-8 &3.03e-9 &1.83e-10 &1.12e-11 &6.95e-13 &$\approx$ 4.02 (4.00)\\
             & 5  &4.94e-9 &1.33e-10 &3.91e-12 &1.19e-13 &3.66e-15 &$\approx$ 5.03 (5.00)\\
             & 6  &3.14e-6 &2.91e-10 &1.02e-13 &1.52e-15 &2.32e-17 &$\approx$ 6.05 (6.00)\\
      \hline
             & 2  &3.03e-5 &7.47e-6 &1.86e-6 &4.63e-7 &1.16e-7 &$\approx$ 2.00 (2.00)\\
             & 3  &1.10e-6 &1.31e-7 &1.59e-8 &1.96e-9 &2.43e-10 &$\approx$ 3.01 (3.00)\\
    $ 0.75$  & 4  &9.98e-8 &5.72e-9 &3.43e-10 &2.10e-11 &1.30e-12 &$\approx$ 4.02 (4.00)\\
             & 5  &1.57e-8 &2.81e-10 &8.24e-12 &2.50e-13 &7.68e-15 &$\approx$ 5.03 (5.00)\\
             & 6  &8.95e-5 &1.61e-8 &2.40e-13 &3.50e-15 &5.33e-17 &$\approx$ 6.07 (6.00)\\
      \hline
     \end{tabular}
\end{table}}

{\small
\begin{table}[htb!]
\caption{The $L^2$-norm error $e^N$ for case (b)
at $t_N=1$, by the uncorrected scheme \eqref{eqn:BDF-CQ-0} with $h=1/100$.}
\label{tab:f-nocorrect}
\centering
     \begin{tabular}{|c|c|ccccc|c|}
     \hline
      $\alpha$ &  $  N$  &$50$ &$100$ &$200$ & $400$ & $800$  &rate \\
      \hline
             & BDF2  &5.14e-4 &2.57e-4 &1.29e-4 &6.45e-5 &3.22e-5 &$\approx$ 1.00 ($1.00$)\\
             & BDF3  &5.19e-4 &2.59e-4 &1.29e-4 &6.45e-5 &3.23e-5 &$\approx$ 1.00 ($1.00$)\\
             & BDF4  &5.18e-4 &2.59e-4 &1.29e-4 &6.45e-5 &3.23e-5 &$\approx$ 1.00 ($1.00$)\\
     $ 0.5$  & BDF5  &5.19e-4 &2.59e-4 &1.29e-4 &6.45e-5 &3.23e-5 &$\approx$ 1.00 ($1.00$)\\
             & BDF6  &5.15e-4 &2.59e-4 &1.29e-4 &6.45e-5 &3.23e-5 &$\approx$ 1.00 ($1.00$)\\
             & L1    &5.98e-4 &2.86e-4 &1.39e-4 &6.80e-5 &3.35e-5 &$\approx$ 1.02 ($1.00$)\\
             & L1-2  &3.71e-4 &1.80e-4 &8.76e-5 &4.25e-5 &2.07e-5 &$\approx$ 1.04 ($--$)\\
                   \hline
     \end{tabular}
\end{table}}

\subsection{Numerical results for diffusion-wave}
Now we illustrate the corrected scheme \eqref{eqn:BDF-CQ-dw} on the following 1D diffusion-wave example:
\begin{itemize}
\item[(c)] $v(x)=x(1-x)$, $b(x)=\sin(2\pi x)$ and $f=e^{t}(1+\chi_{(0,1/2)}(x))$
\end{itemize}

For the diffusion-wave model, the scheme \eqref{eqn:BDF-CQ-dw} is only conditionally stable for
$\alpha\ge \alpha^*(k)= \pi/(\pi-\vartheta_k)$, with a stability threshold $\tau_0= (c(\alpha,k)/r(A))^{1/\alpha}$,
according to Condition \ref{assump:CFL}. To illustrate the sharpness of the threshold $\tau_0$ or equivalently
the CFL constant $c(\alpha,k)$, we consider case (c) with $k=5$, $\alpha=1.5$, $h=1/M=1/100$.
The eigenvalues of the discrete Laplacian $A$ are available in closed form \cite{JinLazarovZhou:SIAM2013}:
\begin{equation*}
 \lambda^h_j= \bar{\lambda}_j^h /(1 - \tfrac{h^2}{6} \bar{\lambda}_j^h), ~~  \quad  \text{with }
  \bar{\lambda}_j^h=-\frac{4}{h^2}\sin^2\frac{\pi j}{2(N+1)},\quad j=1,2,\ldots, M-1.
\end{equation*}
Thus the numerical radius $r(A)=\max_j(\lambda^h_j) \approx 1.2\times 10^5$, which together with
the value $c(\alpha,k)=1.58$ from Fig. \ref{fig:cfl} gives a stability threshold $\tau_0\approx
5.60\times10^{-4}$. In Figs. \ref{fig:R} (a) and (b), we plot the numerical solutions
computed by the corrected scheme \eqref{eqn:BDF-CQ-dw} with $N = 1700$
(i.e., $\tau = 5.88\times10^{-4}$) and with $N = 1800$ (i.e., $\tau = 5.55\times10^{-4}$), respectively.
The scheme \eqref{eqn:BDF-CQ-dw} gives an unstable solution for $N=1700$ but a stable one for $N=1800$. This observation
fully confirms the sharpness of the CFL constant $c(\alpha,k)$ in Condition \ref{assump:CFL}.
In Table \ref{tab:diff-wave-cond}, we present the $L^2$ error for $\alpha>\alpha^*$ and small $\tau$ (such
that it satisfies the CFL condition). The numerical results indicate
the desired $O(\tau^k)$ rate, supporting the theory.

For $\alpha<\alpha^*(k)= \pi/(\pi-\vartheta_k)$, with $\alpha^*$ being the critical value, the corrected scheme
\eqref{eqn:BDF-CQ-dw} based on BDF$k$ is unconditionally stable. Numerically, the corrected scheme \eqref{eqn:BDF-CQ-dw} converges at an
$O(\tau^k)$ rate steadily, cf. Table \ref{tab:diff-wave}, which agrees well with Theorem \ref{thm:conv-dw}.

\begin{figure}[hbt!]
\vspace{-5pt}
\centering
\subfigure[$N=1700$]{
\includegraphics[trim = .1cm .1cm .1cm .1cm, clip=true,width=6.5cm]{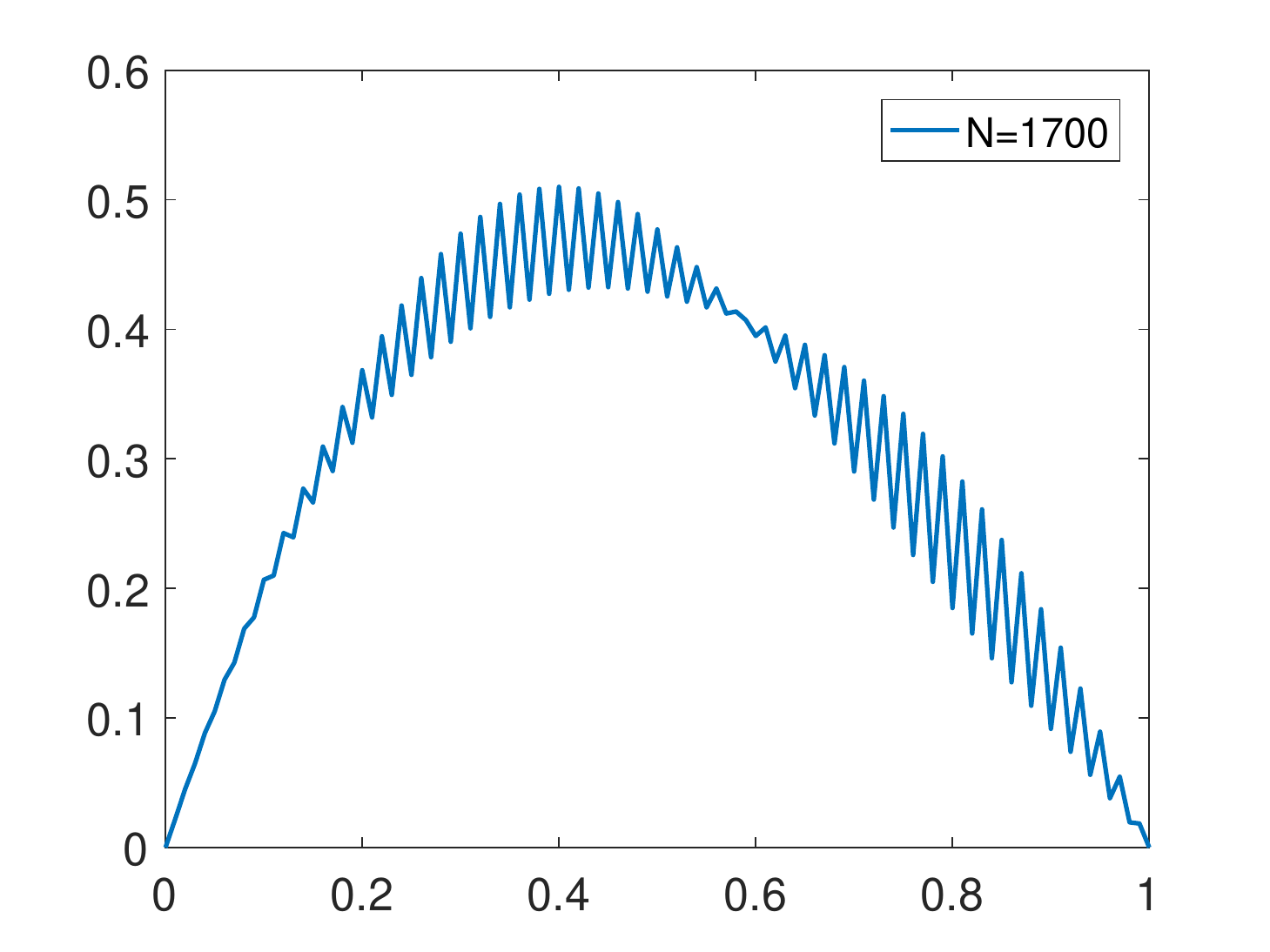}}
\subfigure[$N=1800$]{
\includegraphics[trim = .1cm .1cm .1cm .1cm, clip=true,width=6.5cm]{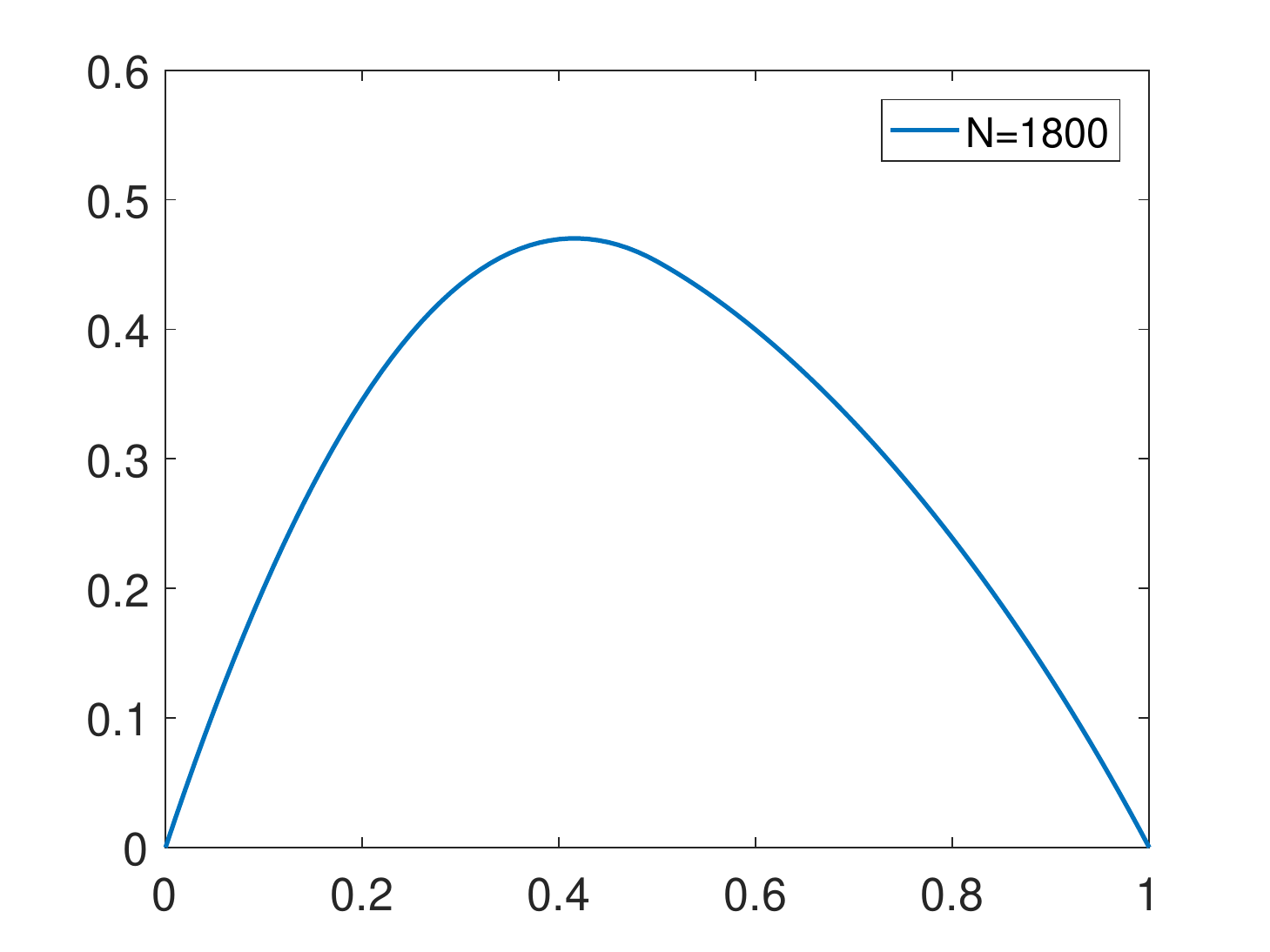}}
 \caption{The numerical solutions for case (c) at $t=1$, with $N=1700$
 ($\tau=5.88\times10^{-4}$) and $N=1800$ ($\tau=5.56\times10^{-4}$), $h=1/100$.
 The theoretical stability threshold is $\tau_0=5.60\times10^{-4}$.\label{fig:R}}
\end{figure}
\vspace{-10pt}

{\small
\begin{table}[htb!]
\caption{The $L^2$-norm error $e^N$ for case (c)
at $t_N=1$,  by the corrected scheme \eqref{eqn:BDF-CQ-dw}, with $h=1/10$.}
\label{tab:diff-wave-cond}
\centering
     \begin{tabular}{|c|c|lllll|c|}
     \hline
      $k$ ($\alpha^*)$ &  $  N$  &$100$ &$200$ &$400$ &$800$ & $1600$  &rate \\
      \hline
       $3$ (1.91)         & 1.95    &2.96e-5 &3.84e-6 &5.00e-7 &6.40e-8 &8.27e-9 &$\approx$ 2.96 ($3.00$)\\
     \hline
       $4 $ (1.68)        & 1.75  &2.08e-6 &1.43e-7 &9.29e-9 &5.92e-10 &3.74e-11 &$\approx$ 3.98 ($4.00$)\\
     \hline
       $5 $ (1.40)        & 1.5  &7.29e-8 &2.49e-10 &6.22e-12 &1.72e-13  &5.05e-15  &$\approx$ 5.14 ($5.00$)\\
     \hline
       $6 $ (1.11)        & 1.5 &5.67e-2  &2.56e-10  &6.88e-13 &1.05e-14  &1.62e-16 &$\approx$ 6.03 ($6.00$)\\
      \hline
     \end{tabular}
\end{table}}

{\small
\begin{table}[htb!]
\caption{The $L^2$-norm error $e^N$ for case (c)
at $t_N=1$,  by the corrected scheme \eqref{eqn:BDF-CQ-dw}, $h=1/100$.}
\label{tab:diff-wave}
\centering
     \begin{tabular}{|c|c|lllll|c|}
     \hline
      $k$ ($\alpha^*)$ &  $  N$  &$100$ &$200$ &$400$ &$800$ & $1600$  &rate \\
      \hline
                       & 1.25                &2.34e-5 &5.85e-6 &1.46e-6 &3.65e-7 &9.14e-8 &$\approx$ 2.00 ($2.00$)\\
       $2$ (2.00)      & 1.5                &6.87e-5 &1.69e-5 &4.18e-6 &1.04e-6 &2.59e-7 &$\approx$ 2.00 ($2.00$)\\
                       & 1.75                 &3.15e-4 &8.55e-5 &2.21e-5 &5.62e-6 &1.42e-6 &$\approx$ 1.98 ($2.00$)\\
     \hline
                          & 1.25  &1.54e-8 &1.66e-9 &3.20e-10 &4.80e-11 &6.33e-12 &$\approx$ 3.00 ($3.00$)\\
       $3$ (1.91)         & 1.5    &4.22e-6 &5.12e-7 &6.30e-8 &7.82e-9 &9.74e-10 &$\approx$ 3.00 ($3.00$)\\
                          & 1.75  &5.27e-5 &6.43e-6 &7.93e-7 &9.78e-8 &1.15e-8 &$\approx$ 3.02 ($3.00$)\\
     \hline
       $4 $ (1.68)        & 1.25  &2.74e-8 &1.64e-9 &1.00e-10 &6.20e-12 &3.63e-13 &$\approx$ 4.00 ($4.00$)\\
                          &1.5    &1.88e-7 &1.27e-8 &8.22e-10 &5.19e-11 &3.07e-12 &$\approx$ 4.00 ($4.00$)\\
                   \hline
       $5 $ (1.40)        & 1.1  &3.32e-10 &9.52e-12 &2.85e-13 &8.71e-15  &2.69e-16  &$\approx$ 5.00 ($5.00$)\\
                          & 1.3  &2.38e-7  &1.28e-10  &1.08e-12  &3.40e-14  &1.06e-15  &$\approx$ 5.00 ($5.00$)\\
                   \hline
       $6 $ (1.11)        & 1.05 &3.31e-5  &1.94e-7  &1.28e-10 &7.58e-17  &7.39e-19 &$\approx$ 6.68 ($6.00$)\\
      \hline
     \end{tabular}
\end{table}}
\vspace{-5pt}

\section*{Acknowledgements}
The authors are grateful to Professor Christian Lubich for his valuable comments on an earlier version of the paper.
The work of B. Jin is supported by UK EPSRC grant EP/M025160/1. The work of B. Li is supported by The Hong
Kong Polytechnic University (A/C code: 1-ZE6L). The work of Z. Zhou is supported in part by the AFOSR  MURI
center for Material Failure Prediction through peridynamics and the ARO MURI Grant W911NF-15-1-0562.

\appendix

\section{An alternative view on the correction scheme \eqref{eqn:BDF-CQ}}\label{app:correction-Lubich}
In this appendix, we discuss the connection between our approach and the approach studied in \cite{Lubich:2004}.
The observation of this connection is due to Professor Christian Lubich.

For the following integral and its convolution quadrature approximation
\begin{equation}\label{eqn:Integral-u}
u(t)=\frac{1}{2\pi \rm i }\int_{\Gamma_{\theta,\delta}}F(z)e^{-tz} d z  \quad\mbox{and}\quad U^n=\frac{1}{2\pi\rm i}\int_{\Gamma_{\theta,\delta}^\tau}\tau F(\delta_\tau(e^{-\tau z}))e^{-t_nz}d z,
\end{equation}
Lubich \cite[Theorem 2.1]{Lubich:2004} showed the following error estimate away from $t=0$:
\begin{align}\label{Error}
|U^n-u(t_n)|\le Ct_n^{\nu-k-1}\tau^k ,
\end{align}
where $\nu\in\mathbb{R}$ is a parameter in the kernel estimate
$|\frac{d^m}{d z^m}F(z)|\le C|z|^{-\nu-m}$, $m\ge 0$.
In particular, if we choose $F(z)=(z^\alpha-A)^{-1}z^{-\ell-1}\partial_t^\ell f(0)$ in \eqref{eqn:Integral-u}, then
\begin{equation*}
   u(t)=\frac{1}{2\pi\rm i}\int_{\Gamma_{\theta,\delta}}(z^\alpha-A)^{-1}z^{-\ell-1}\partial_t^\ell f(0) e^{-tz} d z,\vspace{-10pt}
\end{equation*}\vspace{-10pt}
and
\begin{equation*}
   U^n=\frac{1}{2\pi\rm i}\int_{\Gamma_{\theta,\delta}^\tau} \tau (\delta_\tau(e^{-\tau z})^\alpha-A)^{-1}\delta_\tau(e^{-\tau z})^{-\ell-1}\partial_t^\ell f(0)e^{-t_nz}d z
\end{equation*}
are the integral representations of the solutions of
\begin{align}\label{Original-PDE-}
&\begin{aligned}
    &\partial_t^\alpha u(t)-Au(t) =\frac{t^{\ell}}{\ell !}\partial_t^\ell f(0),\qquad \mbox{with}\quad
    u(0)=0
\end{aligned}\\
&\begin{aligned}\label{eqn:correction-Lubich}
  &\bar\partial_\tau^\alpha U^n - AU^n = \omega_n^{(\ell)}\partial_t^\ell f(0),\qquad \mbox{with}\quad
  U^0=0,
\end{aligned}
\end{align}
respectively, which are solutions and approximations of \eqref{eqn:fde} corresponding to a
single component in the source splitting \eqref{decomp-rhs}.
The weights $\{\omega_n^{(\ell)}\}_{n=0}^\infty$ are the coefficients in the power series expansion
$\delta_\tau(\zeta)^{-\ell-1} = \sum_{n=0}^\infty \omega_n^{(\ell)}\zeta^n .$
By \cite[Theorem 2.1]{Lubich:2004}, the approximation $\{U^n\}$ has the desired accuracy \eqref{Error}.
Our scheme \eqref{eqn:BDF-CQ} is connected to \eqref{eqn:correction-Lubich} as follows:
we replace $\delta(\zeta)^{-\ell-1}$ by an $O(|\zeta-1|^{k-\ell-1})$ close approximation
$\frac{\gamma_\ell(\zeta)}{\ell !} +\sum_{j=1}^{k-1} b_{\ell,j}^{(k)}\zeta^j$, cf. \eqref{crit:b}.
Our choice of the kernel leads to
\begin{align}\label{CQ-eq-new}
\bar\partial_\tau^\alpha U^n- AU^n
= \frac{t_n^{\ell}}{\ell !} \partial_t^\ell f(0)
+  b_{\ell,n}^{(k)}\tau^{\ell}\partial_t^\ell f(0),
\end{align}
which formally differs from \eqref{Original-PDE-} only in the starting $k-1$ steps, as $b_{\ell,n}^{(k)}=0$ for $n\ge k$ in our scheme.
Further, \eqref{CQ-eq-new} is minimal (or optimal) in the sense that it is the unique correction scheme that only
modifies the starting $k-1$ steps while having an accuracy of $O(\tau^k)$.

\section{Proof of Theorem \ref{thm:solurep}}\label{app:sol-rep}

We need the following estimates on the function $\delta_\tau(e^{-z\tau})$.
\begin{lemma}\label{lem:delta}
Let $\alpha\in(0,2)$. For any $\varepsilon$, there exists $\theta_\varepsilon\in (\pi/2,\pi)$ such that for any fixed
$\theta\in (\pi/2,\theta_\varepsilon)$,
there exist positive constants $c,c_1,c_2$ (independent of $\tau$) such that
\begin{equation*}
  \begin{aligned}
& c_1|z|\leq
|\delta_\tau(e^{-z\tau})|\leq c_2|z|,
&&\delta_\tau(e^{-z\tau})\in \Sigma_{\pi-\vartheta_k+\varepsilon}, \\
& |\delta_\tau(e^{-z\tau})-z|\le c\tau^k|z|^{k+1},
&& |\delta_\tau(e^{-z\tau})^\alpha-z^\alpha|\leq c\tau^k|z|^{k+\alpha},
&& \forall\, z\in \Gamma_{\theta,\delta}^\tau .
\end{aligned}
\end{equation*}
\end{lemma}
\begin{proof}
Since the function $\delta(\zeta)/(1-\zeta)$ has no zero in a neighborhood $\mathcal{N}$ of the unit circle
\cite[Proof of Lemma 2]{CreedonMiller:1975} and for $\theta$ sufficiently close to $\pi/2$, $e^{-z\tau}$ lies
in the neighborhood $\mathcal{N}$, there are positive constants $c_1'$ and $c_2'$ such that
$$
c_1'\le \frac{|\delta(e^{-z\tau})| }{|1-e^{-z\tau}| } =\frac{|\delta_\tau(e^{-z\tau})| }{|(1-e^{-z\tau})/\tau | }  \le c_2',\qquad \forall\, z\in \Gamma_{\theta,\delta}^\tau .
$$
Since $\widetilde c_1|z\tau|\le |1-e^{-z\tau}|\le \widetilde c_2 |z\tau|$ for $z\in\Gamma_{\theta,\delta}^\tau$, the first estimate follows.

When $|\zeta|\le 1$ and $\zeta\neq 0$, we have $\delta_\tau(\zeta)\in \Sigma_{\pi-\vartheta_k}$ for the $A(\vartheta_k)$ stable BDF$k$ \cite{HairerWanner:1996}. Hence, by expressing $e^{-z\tau}$ as $e^{-|z|\tau \cos(\theta)}e^{-{\rm i}|z|\tau \sin(\theta)}$, we have
\begin{align*}
|\delta_\tau(e^{-z\tau})-\delta_\tau(e^{-{\rm i}|z|\tau \sin(\theta)})|
&=|\delta_\tau(e^{-|z|\tau \cos(\theta)}e^{-{\rm i}|z|\tau \sin(\theta)})-\delta_\tau(e^{-{\rm i}|z|\tau \sin(\theta)})|\\
&\le ce^{-\sigma|z|\tau \cos(\theta)} \left |\delta_\tau'(e^{-\sigma|z|\tau \cos(\theta)}e^{-{\rm i}|z|\tau \sin(\theta)}) z\tau\cos(\theta)\right|
\end{align*}
for some $\sigma\in(0,1)$, by the mean value theorem.
For $\theta$ close to $\pi/2$ and $z\in\Gamma_{\theta,\delta}^\tau$, by
Taylor expansion, $|z|\tau\leq \pi/\sin\theta$ and the first estimate, we have
$$
\tau |\delta_\tau'(e^{-\sigma|z|\tau \cos(\theta)}e^{-{\rm i}|z|\tau \sin(\theta)})|
\le c\quad \mbox{and}\quad
|\delta_\tau(e^{-{\rm i}|z|\tau \sin(\theta)})|\ge c|z| .
$$
Consequently, we deduce
\begin{align}\label{adjkjk}
|\delta_\tau(e^{-|z|\tau \cos(\theta)}e^{-{\rm i}|z|\tau \sin(\theta)})-\delta_\tau(e^{-{\rm i}|z|\tau \sin(\theta)})|
&\le c|\cos(\theta)| |\delta_\tau(e^{-{\rm i}|z|\tau \sin(\theta)})| \nonumber \\
&\le c|\theta-\pi/2| |\delta_\tau(e^{-{\rm i}|z|\tau \sin(\theta)})| .
\end{align}
Hence, $\delta_\tau(e^{-\tau z})$ is in a sector $\Sigma_{\pi-\vartheta_k+c|\theta-\pi/2|}$. If $\theta>\pi/2$ is sufficiently close to $\pi/2$, then $c|\theta-\pi/2|<\varepsilon$. This proves the second estimate.

The third estimate is given in \cite[eq. (10.6)]{Thomee:2006}.
The last estimate follows from
\begin{align}\label{skkm}
|\delta_\tau(e^{-z\tau})^\alpha-z^\alpha |
=\alpha \left|\int_{z}^{\delta_\tau(e^{-z\tau})} \xi^{\alpha-1}d\xi \right|
\le  \max_{\xi}|\xi|^{\alpha-1}|\delta_\tau(e^{-z\tau})-z| ,
\end{align}
where $\xi$ lies in the line segment with end points $\delta_\tau(e^{-z\tau})$ and $z$. Since
$\Im \delta_\tau(e^{-z\tau})>0$
for $z\in \Gamma_{\theta,\delta}^\tau$ with $\Im z>0$, we have by the first estimate that
$$|\xi|^{\alpha-1}\leq \max(|z|,|\delta_\tau(e^{-z\tau})|)^{\alpha-1}
\leq c|z|^{\alpha-1}.$$
This inequality and \eqref{skkm} yield the last estimate.
\end{proof}

{\it Proof of Theorem {\rm\ref{thm:solurep}}.}$\,\,$
The functions $W^n$, $n=1,\dots,N$,  satisfy (with $W^0=0$):
\begin{equation*}
\begin{aligned}
&\bPtau^\alpha W^n  -  AW^n
=(1+a_n^{(k)})(A v+f(0))+
\sum_{\ell=1}^{k-2}  \bigg(\frac{t_n^\ell}{\ell !}+b_{\ell,n}^{(k)}\tau^{\ell} \bigg)  \partial_t^\ell f (0)
    + R_{k}(t_n) ,
&&1\le n\le k-1,\\[-5pt]
&\bPtau^\alpha W^n  -  AW^n = A v+f(0) +
\sum_{\ell=1}^{k-2}   \frac{t_n^\ell}{\ell !}  \partial_t^\ell f (0) + R_{k}(t_n) , &&k\le n\le N.
\end{aligned}
\end{equation*}
By multiplying both sides by $\zeta ^n$, summing over $n$ and collecting terms, we obtain
\begin{equation*}
\begin{aligned}
&\sum_{n=1}^\infty \zeta ^n \bPtau^\alpha W^n - \sum_{n=1}^\infty AW^n  \zeta ^n\\
&= \bigg( \sum_{n=1}^\infty \zeta ^n+ \sum_{j=1}^{k-1}a_j^{(k)} \zeta ^j \bigg)  (Av+f(0) )
+ \sum_{\ell=1}^{k-2}\bigg(\sum_{n=1}^{\infty} \frac{t_n^\ell}{\ell !}\zeta ^n+ \sum_{j=1}^{k-1} b_{\ell,j}^{(k)}
\tau^{\ell} \zeta ^j\bigg) \partial_t^\ell f (0)+\widetilde R_{k}(\zeta ) \\
&=\bigg(\frac{\zeta }{1-\zeta } + \sum_{j=1}^{k-1}a_j^{(k)}\zeta ^j \bigg)(A v+f(0) )
+\sum_{\ell=1}^{k-2}\bigg(\frac{\gamma_\ell(\zeta)}{\ell !}+
\sum_{j=1}^{k-1} b_{\ell,j}^{(k)} \zeta ^j\bigg)\tau^\ell  \partial_t^\ell f (0)+\widetilde R_{k}(\zeta ) ,
\end{aligned}
\end{equation*}
where $\widetilde R_{k}(\zeta ) = \sum_{n=1}^{\infty} R_{k}(t_n)\zeta ^n$ and elementary identities
\begin{equation}\label{eqn:basic-sum}
  \sum_{n=1}^\infty \zeta ^n = \frac{\zeta }{1-\zeta }\quad \mbox{and}\quad \sum_{n=1}^\infty n^\ell \zeta ^n = \left(\zeta \frac{d}{d\zeta }\right)^\ell\frac{1}{1-\zeta }:=\gamma_\ell(\zeta).
\end{equation}
Next we simplify the summations on both sides. Since $W^0=0$, by the convolution rule,
$\sum_{n=1}^\infty \zeta ^n\bPtau^\alpha W^n=\delta_\tau(\zeta )^\alpha \widetilde{W}(\zeta )$,
and consequently, we obtain
\begin{equation*}
 \begin{aligned}
\widetilde W(\zeta )
&=  K(\delta_\tau(\zeta ))\bigg[\tau^{-1} \mu(\zeta ) (Av + f(0) )
+ \sum_{\ell=1}^{k-2}\delta_\tau(\zeta )\bigg(\frac{\gamma_\ell(\zeta)}{\ell !}+
 \sum_{j=1}^{k-1} b_{\ell,j}^{(k)}\zeta ^j\bigg)\tau^\ell  \partial_t^\ell f (0)+ \delta_\tau(\zeta ) \widetilde R_{k}(\zeta )\bigg] .
\end{aligned}
\end{equation*}
where the operator $K$ is given by \eqref{eqn:kernel}, and the two polynomials $\mu(\zeta)$ and $\gamma_\ell(\zeta)$ are given by \eqref{eqn:gamma-mu}.
Since $\widetilde  W(\zeta )$ is analytic with respect to $\zeta $ in the unit disk on the complex plane,
thus Cauchy's integral formula implies the following representation for arbitrary $\varrho\in(0,1)$
\begin{equation}\label{repr-Wn}
    W^n = \frac{1}{2 \pi\mathrm{i}}\int_{|\zeta |=\varrho}\zeta ^{-n-1} \widetilde  W(\zeta )  d\zeta
    = \frac{\tau}{2\pi\mathrm{i}}\int_{\Gamma^\tau} e^{zt_{n}}\widetilde  W(e^{-z\tau})\, dz,
\end{equation}
where the second equality follows from the change of variable $\zeta =e^{-z\tau}$, and $\Gamma^\tau$ is given by
\begin{equation*}
   \Gamma^\tau:=\{ z=-\ln(\varrho)/\tau+\mathrm{i} y: \, y\in{\mathbb R}\,\,\,\mbox{and}\,\,\,|y|\le {\pi}/{\tau} \}.
\end{equation*}
Note that
\begin{itemize}
\item[(1)] $\eta(\zeta):=\delta_\tau(\zeta)/(1-\zeta)$ is a polynomial without roots in a neighborhood $\mathcal{N}$ of the unit circle \cite{CreedonMiller:1975}. Thus, $\eta(\zeta)^\alpha$ is analytic in $\mathcal{N}$.
\item[(2)] By choosing the angle $\theta$ sufficiently close to $\pi/2$, $\varrho$ sufficiently close to $1$ and $0<\delta<-\ln(\varrho/\tau)$,
the function $e^{-\tau z}$ lies in $\mathcal{N}$ for
$$z\in \Sigma_{\theta,\delta}^\tau=\{z\in\Sigma_{\theta}:|z|\ge\delta,\,\, |{\rm Im}(z)|\le \tau/\pi,\,\,\, {\rm Re}(z)\le -\ln(\varrho)/\tau \} ;$$
\item[(3)] $(1-e^{-\tau z})^\alpha$ is analytic for $z\in {\mathbb C}\backslash(-\infty,0] \supset  \Sigma_{\theta,\delta}^\tau $.
\end{itemize}
These properties ensure that $\delta_\tau(e^{-\tau z})^\alpha=\tau^{-\alpha}(1-e^{-\tau z})^\alpha\eta(e^{-\tau z})^\alpha$ is analytic for $z\in \Sigma_{\theta,\delta}^\tau$.
By choosing $\varepsilon$ small enough, Lemma \ref{lem:delta} implies that $0\neq \delta_\tau(e^{-\tau z})^\alpha\in \Sigma_{\alpha(\vartheta_k+\varepsilon)}\subset \Sigma_{\pi-\varepsilon}$ for
$z\in \Sigma_{\theta,\delta}^\tau$. Thus $K(\delta_\tau(e^{-\tau z}))=\delta_\tau(e^{-\tau z})^{-1}(\delta_\tau(e^{-\tau z})^\alpha-A)^{-1}$ is analytic for $z\in \Sigma_{\theta,\delta}^\tau$, which is a region enclosed by $\Gamma^\tau$, $\Gamma^\tau_{\theta,\delta}$ and the two lines $\Gamma_{\pm}^\tau:
={\mathbb R}\pm \mathrm{i}\pi/\tau$ (oriented from left to right).
Since the values of $e^{zt_{n}}\widetilde  W
(e^{-z\tau})$ on $\Gamma_{\pm}^\tau$ coincide, Cauchy's theorem allows deforming the contour $\Gamma^\tau$
to $\Gamma_{\theta,\delta}^\tau$ in the integral \eqref{repr-Wn} to obtain the desired representation.
\endproof

\section{Proof of Theorem \ref{thm:conv}}\label{app:conv-sub}

\begin{lemma}\label{lem:kernel}
Let Criteria \eqref{crit:mu}  and \eqref{crit:b} hold. Then for $z\in \Gamma_{\theta,\delta}^\tau$, there hold
\begin{align}
    \|\mu(e^{-z\tau})K(\delta_\tau(e^{-z\tau})) -K(z)\| & \leq  c \tau^k |z|^{k-1-\alpha},\label{eqn:est-Kmu}\\
    \bigg\|(\delta_\tau(e^{-z\tau})^\alpha-A)^{-1} \bigg(\frac{1}{\ell !}\gamma_\ell(e^{-z\tau})+ \sum_{j=1}^{k-1} b_{\ell,j}^{(k)} e^{-jz\tau}\bigg)\tau^{\ell+1}- z^{-\ell}K(z) \bigg\| & \leq c \tau^k |z|^{k-\ell-1-\alpha}.\label{eqn:est-Kb}
  \end{align}
\end{lemma}

{\it Proof.}$\,\,$
Since $|1-e^{-z\tau}|\leq c\tau|z|$ for $z\in \Gamma_{\theta,\delta}^\tau$,
by Criterion \eqref{crit:mu}, there holds
$  |\mu(e^{-z\tau})-1|\leq c|1-e^{-z\tau}|^k\leq c \tau^k |z|^k.$
Meanwhile, by the triangle inequality, we have
\begin{equation*}
  \begin{aligned}
  \|K(\delta_\tau(e^{-z\tau}))-K(z)\| & = \|\delta_\tau(e^{-z\tau})^{-1}(\delta_\tau(e^{-z\tau})^\alpha-A)^{-1}-z^{-1}(z^\alpha-A)^{-1}\|\\
  & \leq |\delta_\tau(e^{-z\tau})^{-1}-z^{-1}|\|(\delta_\tau(e^{-z\tau})^\alpha-A)^{-1}\| \\
  &\qquad+ |z|^{-1}\|(\delta_\tau(e^{-z\tau})^\alpha-A)^{-1}-(z^\alpha-A)^{-1}\|.
\end{aligned}
\end{equation*}
The identity $(\delta_\tau(e^{-z\tau})^\alpha-A)^{-1}-(z^\alpha-A)^{-1}=
(z^\alpha-\delta_\tau(e^{-z\tau})^\alpha)(\delta_\tau(e^{-z\tau})^\alpha-A)^{-1}(z^\alpha-A)^{-1}$,
Lemma \ref{lem:delta} and the resolvent estimate \eqref{eqn:resol} imply directly
$ \|K(\delta_\tau(e^{-z\tau}))-K(z)\| \leq c|\tau|^k |z|^{k-1-\alpha}.$
Consequently, we obtain the estimate \eqref{eqn:est-Kmu} by
\begin{equation*}
  \begin{aligned}
    \|\mu(e^{-z\tau})K(\delta_\tau(e^{-z\tau})) &-K(z)\|  \leq|\mu(e^{-z\tau})-1|\|K(\delta_\tau(e^{-z\tau}))\| \\
      & + \|K(\delta_\tau(e^{-z\tau}))-K(z)\|  \le  c \tau^k |z|^{k-1-\alpha}\quad \forall z\in \Gamma_{\theta,\delta}^\tau.
  \end{aligned}
\end{equation*}
Next we show the estimate \eqref{eqn:est-Kb}. By Lemma \ref{lem:delta}, there holds
\begin{equation*}
  |\delta_\tau(e^{-z\tau})^{\ell+1}-z^{\ell+1}| \leq c|\delta_\tau(e^{-z\tau})-z||z|^\ell \leq c \tau^k |z|^{k+\ell+1} \quad \forall z\in\Gamma_{\theta,\delta}^\tau.
\end{equation*}
By Criterion \eqref{crit:b}, there holds
\begin{equation*}
\bigg|\frac{\gamma_\ell(e^{-z\tau})}{\ell !}+ \sum_{j=1}^{k-1} b_{\ell,j}^{(k)} e^{-jz\tau} - \frac{1}{\delta(e^{-z\tau})^{\ell+1}}\bigg|\leq c\tau^{k-\ell-1}|z|^{k-\ell-1} \quad \forall z\in \Gamma_{\theta,\delta}^\tau.
\end{equation*}
Consequently, for any $z\in \Gamma_{\theta,\delta}^\tau$, we have
\begin{equation*}
\begin{aligned}
  & \bigg\|(\delta_\tau(e^{-z\tau})^\alpha-A)^{-1} \bigg(\frac{1}{\ell !}\gamma_\ell(e^{-z\tau})+ \sum_{j=1}^{k-1} b_{\ell,j}^{(k)} e^{-jz\tau}\bigg)\tau^{\ell+1}- z^{-\ell}K(z) \bigg\|\\[-5pt]
  \leq & \bigg\|(\delta_\tau(e^{-z\tau})^\alpha-A)^{-1}\bigg[\bigg(\frac{1}{\ell !}\gamma_\ell(e^{-z\tau})+ \sum_{j=1}^{k-1} b_{\ell,j}^{(k)} e^{-jz\tau}\bigg)\tau^{\ell+1}-\delta_\tau(e^{-z\tau})^{-\ell-1}\bigg]\bigg\| \\
   & + \|\delta_\tau(e^{-z\tau})^{-\ell}K(\delta_\tau(e^{-z\tau}))-z^{-\ell}K(z)\|\leq c \tau^k |z|^{k-\ell-1-\alpha}.
\end{aligned}
\end{equation*}
This completes the proof of the lemma.
\endproof


{\it Proof of Theorem \ref{thm:conv}.}$\,\,$
By \eqref{eqn:semisol} and \eqref{eqn:Rep-Wh}, we split $U^n-u(t_n)=W^n-w(t_n)$ into
\begin{equation*}
  W^n - w(t_n) = I_1 + \sum_{\ell=1}^{k-2} I_{2,\ell} -I_{3}  + I_4,
\end{equation*}
where the terms $I_1,\ldots,I_4$ are given by
\begin{equation*}
  \begin{aligned}
    I_1=&\frac{1}{2\pi \mathrm{i}}\int_{\Gamma^\tau_{\theta,\delta}}
  e^{zt_n}\Big(\mu(e^{-z\tau})K(\delta_\tau(e^{-z\tau}))-K(z)\Big)(Av  +f(0)) dz,\\
   I_{2,\ell}= & \frac{1}{2\pi\mathrm{i} }\int_{\Gamma^\tau_{\theta,\delta}}e^{zt_{n}}\bigg[\delta_\tau(e^{-z\tau})
\bigg(\frac{\gamma_\ell(e^{-z\tau})}{\ell !}+  \sum_{j=1}^{k-1} b_{\ell,j}^{(k)} e^{-z\tau j}\bigg)\tau^{\ell+1}  K(\delta_\tau(e^{-z\tau}))-z^{-\ell}K(z)\bigg] \partial_t^\ell f (0)\,dz,\\
I_{3}= &\frac{1}{2\pi\mathrm{i} }\int_{\Gamma_{\theta,\delta}\backslash\Gamma^\tau_{\theta,\delta}}e^{zt_{n}}K(z)
\big(Av  +f(0) + z^{-\ell}  \sum_{j=1}^{k-1}\partial_t^\ell f (0) \big)\,dz,\\
I_4 = &\frac{1}{2\pi \mathrm{i}}\int_{\Gamma^\tau_{\theta,\delta}}
  e^{zt}  (\delta_\tau(e^{-z\tau})^\alpha-A)^{-1} \tau \widetilde  R_{k}(e^{-z\tau})dz
  -\frac{1}{2\pi \mathrm{i}}\int_{\Gamma_{\theta,\delta}} (z^\alpha-A)^{-1}\widehat R_{k}(z) dz.
  \end{aligned}
\end{equation*}
It suffices to bound these terms separately. By Lemma \ref{lem:kernel},
and choosing $\delta=t_n^{-1}$ in the contour $\Gamma_{\theta,\delta}^\tau$, we bound the first term $I_1$ by
\begin{equation*}
  \begin{aligned}
      \| I_1\|_{L^2(\Omega)} &\le c \tau^k \| Av+f(0) \|_{L^2(\Omega)} \bigg(\int_{\delta}^{\pi/(\tau\sin\theta)}  e^{r t_n\cos\theta} r^{k-1-\alpha}dr
     + \int_{-\theta}^{\theta}   e^{\delta t_n|\cos\psi|} \delta^{k-\alpha}d\psi \bigg)\\
     &\le c \tau^k(t_n^{\alpha-k}+\delta^{k-\alpha}) \| Av+f(0) \|_{L^2(\Omega)} \le c \tau^k t_n^{\alpha-k}  \|Av+f(0) \|_{L^2(\Omega)}.
  \end{aligned}
\end{equation*}
By appealing to Lemma \ref{lem:kernel} again and choosing $\delta=t_n^{-1}$ in
$\Gamma_{\theta,\delta}^\tau$, we bound the terms $I_{2,\ell}$ by
\begin{align*}
\|I_{2,\ell}\|_{L^2(\Omega)}
    &\le c\tau^k \|f^{(\ell)}\|_{L^2(\Omega)}\bigg(\int_{\delta}^{\pi/(\tau\sin\theta)}  e^{r t_n\cos\theta} r^{k-\ell-1-\alpha} dr
     + \int_{-\theta}^{\theta}  e^{\delta t_n|\cos\psi|} \delta^{k-\ell-\alpha} d\psi \bigg) \\
     &\le c \tau^k t_n^{ \alpha+\ell -k }  \| \partial_t^\ell f (0)\|_{L^2(\Omega)},\quad \ell=1,2...,k-1.
\end{align*}
Direct computation yields the following estimate on $I_3$:
\begin{equation*}
  \|I_3\|_{L^2(\Omega)} \le  c \tau^k \bigg(t_n^{\alpha-k} \|Av+f(0) \|_{L^2(\Omega)} +  \sum_{\ell=1}^{k-2} t_n^{ \alpha+\ell -k }  \| \partial_t^\ell f (0)\|_{L^2(\Omega)} \bigg).
\end{equation*}
The term $I_4$ is the error of the numerical solution with a compatible right-hand side $R_{k}$.
Upon recalling the definition of $R_k$ in \eqref{eqn:Rk}, we use the splitting
$R_k  = \frac{t^{k-1}}{(k-1)!} \partial_t^{k-1} f(0) + \frac{t^{k-1}}{(k-1)!}*\partial_t^k f (t) =:  R_k^{1} +  R_k^{2}.$
Then we have $I_4=  I_4^1 + I_4^2$ with
\begin{equation*}
I_4^i = \frac{1}{2\pi \mathrm{i}}\int_{\Gamma^\tau_{\theta,\delta}}
  e^{zt_n}  (\delta_\tau(e^{-z\tau})^\alpha-A)^{-1} \tau \widetilde  R_{k}^i (e^{-z\tau})dz
  -\frac{1}{2\pi \mathrm{i}}\int_{\Gamma_{\theta,\delta}}e^{zt_n} (z^\alpha-A)^{-1}\widehat R_{k}^i (z) dz.
\end{equation*}
By repeating the preceding argument and \eqref{Taylor-gln}, we have the estimate for $I_4^1$:
\begin{equation*}
  \|  I_4^1 \|_{L^2(\Omega)} \le c \tau^k t_n^{\al-1} \| \partial_t^{k-1}f(0)\|_{L^2(\Omega)},
\end{equation*}
and using the argument in \cite[Lemma 3.7]{JinLiZhou:CN},
\begin{equation*}
  \|  I_4^2 \|_{L^2(\Omega)} \le c \tau^k \int_0^{t_n}(t_n-s)^{\alpha-1}\| \partial_s^k f(s)\|_{L^2(\Omega)}ds.
\end{equation*}
This completes the proof of the theorem.
\endproof

\section{Proof of Theorem \ref{thm:solurep-dw}}\label{app:sol-rep-dw}
Using the splitting \eqref{eqn:decomp-rhs-dw}, the functions $W^n$, $n=1,\dots,N$, satisfy (with $W^0=0$):
\begin{equation*}
\begin{aligned}
&\bPtau^\alpha W^n  -  AW^n
=(1+a_n^{(k)})A v+ (t_n+\tau c_n^{(k)})Ab +
\sum_{\ell=1}^{k-2}  \bigg(\frac{\bar\partial_\tau t_n^\ell}{\ell !}+b_{\ell,n}^{(k)}\tau^{\ell-1}\bigg) \partial_t^{\ell-1}f(0)
    + \bar\partial_\tau R_{k}(t_n) , \\[-5pt]
&\qquad\qquad\qquad\qquad\qquad\qquad\qquad\qquad\qquad
\qquad\qquad\qquad\qquad\qquad\qquad\qquad\qquad 1\le n\le k-1,\\[-5pt]
&\bPtau^\alpha W^n  -  AW^n = A v+t_nAb +
\sum_{\ell=1}^{k-2}   \frac{\bar\partial_\tau t_n^\ell}{\ell !}\partial_t^{\ell-1}f(0) + \bar\partial_\tau R_{k}(t_n) ,
\qquad\qquad\qquad\qquad\quad
k\le n\le N.
\end{aligned}
\end{equation*}
By multiplying both sides by $\zeta ^n$ and summing over $n$, we obtain
\begin{equation*}
\begin{aligned}
\sum_{n=1}^\infty \zeta ^n \bPtau^\alpha W^n - \sum_{n=1}^\infty AW^n  \zeta ^n&=\bigg(\sum_{n=1}^\infty \zeta^n + \sum_{j=1}^{k-1}a_j^{(k)}\zeta ^j \bigg)A v
+\bigg(\sum_{n=1}^\infty \tau n\zeta^n + \sum_{j=1}^{k-1}\tau c_j^{(k)}\zeta^j\bigg)Ab\\
 & \quad + \sum_{\ell=1}^{k-2}\bigg(\sum_{n=1}^\infty \frac{\bar\partial_\tau t_n^\ell}{\ell !}\zeta_n^n +
\sum_{j=1}^{k-1} b_{\ell,j}^{(k)}\tau^{\ell-1} \zeta ^j\bigg) \partial_t^{\ell-1}f(0)+ \sum_{n=1}^\infty \bar\partial_\tau R_{k}(t_n)\zeta^n.
\end{aligned}
\end{equation*}
Using the elementary identities in \eqref{eqn:basic-sum}, the convolution rule
$\sum_{n=1}^\infty \zeta ^n \bPtau^\alpha W^n = \delta_\tau(\zeta)^\alpha\widetilde W$, and
$
 \sum_{n=1}^\infty \zeta^n \bar\partial_\tau \frac{t_n^\ell}{\ell!} = \delta_\tau(\zeta) \sum_{n=0}^\infty \frac{t_n^\ell}{\ell!}\zeta^n = \delta(\zeta)\frac{\tau^{\ell-1}}{\ell!}\gamma_{\ell}(\zeta),
$
we derive
\begin{equation*}
 \begin{aligned}
\widetilde W(\zeta )
&=  K(\delta_\tau(\zeta ))\left[\tau^{-1} \mu(\zeta ) Av + \delta_\tau(\zeta)\bigg(\gamma_1(\zeta)+\sum_{j=1}^{k-1}c_j^{(k)}\zeta^j\bigg)\tau Ab\right.\\
 &\quad \left.+ \sum_{\ell=1}^{k-2}\delta_\tau(\zeta )\bigg(\delta(\zeta)\frac{\gamma_\ell(\zeta)}{\ell !} + \sum_{j=1}^{k-1} b_{\ell,j}^{(k)}\zeta ^j\bigg)\tau^{\ell-1} g^{(\ell)}+ \delta_\tau(\zeta)^2 \widetilde R_{k}(\zeta )\right] .
\end{aligned}
\end{equation*}
{Under Condition \ref{assump:CFL} (i),
by choosing $\varepsilon$ small enough, Lemma \ref{lem:delta} implies that $0\neq \delta_\tau(e^{-\tau z})^\alpha\in \Sigma_{\alpha(\vartheta_k+\varepsilon)}\subset \Sigma_{\pi-\varepsilon}$ for
$z\in \Sigma_{\theta,\delta}^\tau$.
Under Condition \ref{assump:CFL} (ii), we have
$
{\rm dist}(\delta(e^{-z\tau})^\alpha,\tau^\alpha  S(A))>0
$ (cf. Appendix \ref{app:diff-wave}),
where $S(A)$ denotes the closure of the spectrum of $A$ in the complex plane $\mathbb{C}$.
In either case, the operator $K(\delta_\tau(e^{-\tau z}))=\delta_\tau(e^{-\tau z})^{-1}(\delta_\tau(e^{-\tau z})^\alpha-A)^{-1}$ is analytic for $z\in \Sigma_{\theta,\delta}^\tau$, which is the region enclosed by the four curves $\Gamma^\tau_{\theta,\delta}$, $-\ln(\varrho)/\tau+i{\mathbb R}$
and ${\mathbb R}\pm \mathrm{i}\pi/\tau$ (for $\theta$ and $\varrho$ sufficiently close to $\pi/2$ and $1$, respectively).}
Then, like in the proof of Theorem \ref{thm:solurep}, the assertion follows from Cauchy's integral formula and the change of variables $\zeta = e^{-z\tau}$.
\endproof

\section{Proof of Theorem \ref{thm:conv-dw}}\label{app:diff-wave}
Under Condition \ref{assump:CFL}(i), Theorem \ref{thm:conv-dw} can be proved in the same way as Theorem \ref{thm:conv},
using \eqref{eqn:sol-rep-dw-cont} and \eqref{eqn:sol-rep-dw-dis}.
Under Condition \ref{assump:CFL}(ii), it can be proved analogously,
provided that the following resolvent estimate holds:
\begin{align}\label{resolvent-CFL}
\|(\delta_\tau(e^{-z\tau})^\alpha-A)^{-1}\|
\le c |z|^{-\alpha} ,\quad \forall\, z\in \Gamma^\tau_{\theta,\delta}.
\end{align}
To prove \eqref{resolvent-CFL}, we use the following estimate (cf. \cite[Theorem 3.9, Chapter 1, pp. 12]{Pazy:1983}):
\begin{align}\label{djhh}
\|(\delta_\tau(e^{-z\tau})^\alpha-A)^{-1}\|
&=\tau^\alpha\|(\delta(e^{-z\tau})^\alpha-\tau^\alpha A)^{-1}\| \nonumber\\
&\le c\tau^\alpha\, {\rm dist}(\delta(e^{-z\tau})^\alpha,\tau^\alpha  S(A))^{-1},\quad \forall\, z\in \Gamma^\tau_{\theta,\delta} ,
\end{align}
where $S(A)$ denotes the closure of the spectrum of $A$ in $\mathbb{C}$.
For the discrete Laplacian $A=\Delta_h$, we have $S(A)=[-r(A),0]$. Since the angle between the contour
$\delta(e^{-{\rm i}\xi\tau})^\alpha$, $\xi\in[-\frac{\pi}{\tau},\frac{\pi}{\tau}]$, and the segment
$[-r(A),0]$ is $(1-\alpha/2)\pi>0$,
it follows that, for small $\kappa$,
$$
{\rm dist}(\delta(e^{-{\rm i}\xi\tau})^\alpha,\tau^\alpha S(A))
\approx
|\delta(e^{-{\rm i}\xi\tau})^\alpha|\sin[(1-\alpha/2)\pi]
\ge
c|\delta(e^{-{\rm i}\xi\tau})^\alpha|
\quad\mbox{if}\,\,\, |\xi|\tau \le\kappa .
$$
Furthermore, CFL condition \ref{assump:CFL}(ii) implies
\begin{equation*}
{\rm dist}(\delta(e^{-{\rm i}\xi\tau})^\alpha,\tau^\alpha S(A))
\ge c \ge c |\xi\tau|^\alpha \quad\mbox{if}\,\,\, \kappa\le |\xi|\tau  \le \pi .
\end{equation*}
Let $\Gamma^\tau_\theta =\{z\in{\mathbb C}:{\rm arg}(z)
=\theta,\,\,-\frac{\pi}{\tau}\le |z|\sin(\theta) \le \frac{\pi}{\tau }\}.$
Then the angle between the contour $\delta(e^{-z\tau})^\alpha$, $z\in\Gamma^\tau_\theta$,
and the segment $[-r(A),0]$ is $\pi-\alpha\theta>0$ (if $\theta$ is close to $\pi/2$).
For small $\kappa$ and $z\in\Gamma^\tau_\theta$, $|z|\tau  \sin(\theta) \le\kappa$, we have
\begin{equation*}
{\rm dist}(\delta(e^{-z\tau})^\alpha,\tau^\alpha S(A))
\approx |\delta(e^{-z\tau})^\alpha|\sin(\pi-\alpha\theta) \ge
c |\delta(e^{-z\tau})^\alpha| \ge c|z\tau|^{\alpha} .
\end{equation*}
Estimate \eqref{adjkjk} implies
$|\delta(e^{-z\tau})-\delta(e^{-{\rm i}|z|\tau \sin(\theta)})|\le c|\theta-\pi/2| ,$
and thus
\begin{align*}
|\delta(e^{-z\tau})^\alpha-\delta(e^{-{\rm i}|z|\tau \sin(\theta)})^\alpha|
&\le c|\theta-\pi/2| \min(|\delta(e^{-z\tau})|^{\alpha-1},|\delta(e^{-{\rm i}|z|\tau \sin(\theta)}|^{\alpha-1})\\
&\le c|\theta-\pi/2| |z\tau|^{\alpha-1} ,\qquad z\in\Gamma_\theta^\tau .
\end{align*}
Hence, if $z\in\Gamma_\theta^\tau$ and $\kappa\le |z|\tau  \sin(\theta) \le\pi$,
with $\theta$ close to $\pi/2$, we have
\begin{align*}
{\rm dist}(\delta(e^{-z\tau})^\alpha,\tau^\alpha S(A))
&\ge {\rm dist}(\delta(e^{-{\rm i}|z|\tau \sin(\theta)})^\alpha,\tau^\alpha S(A))
-|\delta(e^{-z\tau})^\alpha-\delta(e^{-{\rm i}|z|\tau \sin(\theta)})^\alpha|\\
&\ge c  - c|\theta-\pi/2| |z\tau|^{\alpha-1} \ge c - c|\theta-\pi/2| |z\tau\sin(\theta)|^{\alpha-1} \\
&\ge c - c|\theta-\pi/2| \max(\kappa,\pi)^{\alpha-1}
\ge c  \ge c|z\tau|^\alpha .
\end{align*}
Thus we have ${\rm dist}(\delta(e^{-z\tau})^\alpha,\tau^\alpha S(A))\ge c|z\tau|^\alpha$
for $z\in\Gamma_\theta^\tau$. This inequality and \eqref{djhh} yield
\eqref{resolvent-CFL} for $z\in\Gamma^\tau_{\theta,\delta}\cap\Gamma^\tau_\theta$.
Further, if $z\in\Gamma^\tau_{\theta,\delta}\backslash\Gamma^\tau_\theta$, then
$|z|=\delta$ and $-\theta<{\rm arg}(z)<\theta$, and Taylor expansion yields
$|\delta(e^{-z\tau})|^\alpha\le |z\tau|^\alpha \le \delta^\alpha\tau^\alpha .$
By choosing $\delta$ small, we have
\begin{equation*}
{\rm dist}(\delta(e^{-z\tau})^\alpha,\tau^\alpha S(A)) \ge
\lambda_{\min}\tau^\alpha-\delta^\alpha\tau^\alpha
\ge c\tau^\alpha ,
\end{equation*}
where $\lambda_{\min}$ is the smallest positive eigenvalue of the operator $A$ (which
can be made independent of $h$). This and \eqref{djhh}  yield
\begin{equation*}
\|(\delta_\tau(e^{-z\tau})^\alpha-A)^{-1}\|
\le c \le c\delta^{-\alpha}
=c|z|^{-\alpha}  ,\quad \forall\, z\in\Gamma^\tau_{\theta,\delta}\backslash\Gamma^\tau_\theta .
\end{equation*}
This completes the proof of \eqref{resolvent-CFL}.
\endproof

\bibliographystyle{abbrv}
\bibliography{frac}

\end{document}